\documentclass[11pt]{amsart}
\usepackage{amsmath,amsfonts,amsthm,amssymb}
\usepackage{amscd}

\usepackage{array}
\usepackage[all]{xy}
\usepackage{epsfig}
\usepackage{graphicx}
\usepackage{color}

\usepackage{euscript}

\definecolor{white}{rgb}{1,1,1}\long\def\symbolfootnote[#1]#2{\begingroup\def\thefootnote{\fnsymbol{footnote}}\footnote[#1]{#2}\endgroup} 

\newcounter{thecounter}
\numberwithin{thecounter}{section}
\newtheorem{lemma}[thecounter]{Lemma}
\newtheorem{proposition}[thecounter]{Proposition}
\newtheorem{theorem}[thecounter]{Theorem}
\newtheorem{corollary}[thecounter]{Corollary}
\newtheorem{defn}[thecounter]{Definition}

\begin{document}

\title{Classification of hyperbolic Dynkin diagrams, root lengths \\ and Weyl group orbits}

\date{\today}

\author{Lisa Carbone}
\author{Sjuvon Chung}
\author{Leigh Cobbs}
\author{Robert McRae}
\author{Debajyoti Nandi}
\author{Yusra Naqvi}
\author{Diego Penta}

\thanks{The first author was supported in part by NSF grant \#DMS-0701176.}

\address{Department of Mathematics, Hill Center, Busch Campus\\
Rutgers, The State University of New Jersey\\
110 Frelinghuysen Rd\\
Piscataway, NJ 08854-8019}
\email{carbonel@math.rutgers.edu}
\email{sjuvon@gmail.com}
\email{cobbs@math.rutgers.edu}
\email{rhmcrae@math.rutgers.edu}
\email{nandi@math.rutgers.edu}
\email{ynaqvi@math.rutgers.edu}
\email{penta@math.rutgers.edu}

\begin{abstract} We give a criterion for a Dynkin diagram, equivalently a generalized Cartan matrix, to be symmetrizable. This criterion is easily checked on the Dynkin diagram. We obtain a simple proof  that the maximal rank of a Dynkin diagram of compact hyperbolic type is 5, while the maximal rank of a symmetrizable Dynkin diagram of compact hyperbolic type is 4. Building on earlier classification results of Kac, Kobayashi-Morita, Li and Sa\c{c}lio\~{g}lu, we present the 238 hyperbolic Dynkin diagrams in ranks 3-10, 142 of which are symmetrizable. For each symmetrizable hyperbolic generalized Cartan matrix, we give a symmetrization and hence the distinct lengths of real roots in the corresponding root system. For each such hyperbolic root system we determine the disjoint orbits of the action of the Weyl group on real roots. It follows that the maximal number of disjoint Weyl group orbits on real roots in  a hyperbolic root system is 4.

\end{abstract}

\maketitle

\section{Introduction} \label{intro}

\medskip
\noindent The theory of hyperbolic Kac-Moody groups and algebras naturally generalizes the theory of affine Kac-Moody groups and algebras which is itself the most natural generalization to infinite dimensions of finite dimensional Lie theory. Recently,  hyperbolic and Lorentzian Kac-Moody groups and algebras have been discovered as symmetries of dimensionally reduced supergravity theories and are known to parametrize the scalar fields of supergravity theories via their coset spaces. They are conjectured to be symmetries of  full supergravity theories ([DHN], [We]) and to encode  geometrical objects of M-theory ([BGH], [DHN]) as well as the dynamics of certain supergravity theories near a space-like singularity ([DHN]).

\medskip
\noindent It is desirable then to have a clear statement of the classification of hyperbolic Kac-Moody  algebras, which may be obtained by classifying their Dynkin diagrams. This is achieved  in analogy with the Cartan-Killing classification of finite dimensional simple Lie algebras which can be described in terms of a classification of their Dynkin diagrams.

\medskip
\noindent If a Dynkin diagram $\mathcal{D}$  is neither of finite or affine type, $\mathcal{D}$ is called {\it indefinite}. The hyperbolic Dynkin diagrams and their Kac-Moody algebras are an important subclass of the indefinite class. We recall that  $\mathcal{D}$ is of {\it hyperbolic type} if it is neither of finite nor affine type, but every proper connected subdiagram is either of finite or of affine type. If $\mathcal{D}$ is of  hyperbolic type,  we say that $\mathcal{D}$ is of {\it compact hyperbolic type} if every proper, connected subdiagram is of finite type.

\medskip
\noindent It is known that the maximal rank of a hyperbolic Kac-Moody algebra is 10. This is determined by the following restrictive conditions:

\medskip
\noindent (1) The fundamental chamber $\mathcal{C}$ of the Weyl group, viewed as a hyperbolic reflection group, must be  a Coxeter polyhedron. The dihedral angles between adjacent walls must be of the form $\pi/k$, where $k\geq 2$.

\medskip
\noindent  (2) The fundamental chamber $\mathcal{C}$ of the Weyl group must be a simplex, which gives a bound on the number of faces.

\medskip
\noindent Such a `Coxeter simplex' $\mathcal{C}$ exists in hyperbolic $n$-space $\mathbb{H}^n$ for $n\leq 9$ (see [VS]). The bound on the rank of a hyperbolic Dynkin diagram can also be deduced by purely combinatorial means ([KM], [K], [Li], [Sa]). Thus the maximal rank of a  hyperbolic Kac-Moody algebra is 10 with a Lorentzian root space of signature $(9,1)$.

\medskip
\noindent   The data for constructing a Kac-Moody algebra includes a {\it generalized Cartan matrix} which is a generalization of the notion of a Cartan matrix for finite dimensional Lie algebras, and which encodes the same information as a Dynkin diagram. Given a generalized Cartan matrix, or its Dynkin diagram, and a finite dimensional vector space $\mathfrak{h}$ satisfying some natural conditions, Gabber and Kac  defined a Kac-Moody algebra by generators and relations in analogy with the Serre presentation for finite dimensional simple Lie algebras ([GK]).

\medskip
\noindent `Symmetrizability' is an important property of the generalized Cartan matrix of a Kac-Moody algebra, necessary for the existence of a well-defined symmetric invariant bilinear form $(\cdot\mid\cdot)$ on the Kac-Moody algebra. This   invariant bilinear form  plays the role of `squared length' of a root. When a generalized Cartan matrix is not symmetrizable, a Kac-Moody algebra may still be constructed, though one must keep track of the discrepancies in the definition of $(\cdot\mid\cdot)$.

\medskip
\noindent One of our motivations in completing this work was to understand the appearance of hyperbolic Dynkin diagrams in cosmological billiards.  In [dBS], de Buyl and Schomblond identify the hyperbolic Kac-Moody algebras for which there exists a Lagrangian of gravity, dilatons and $p$-forms which gives rise to a billiard that can be identified with the fundamental chamber of the Weyl group of the Kac-Moody algebra. In [HJ], Henneaux and Julia compute the billiards that emerge  for all pure supergravities in $D=4$ spacetime dimensions, as well as for $D=4$, $N=4$ supergravities coupled to $k$  Maxwell supermultiplets. They show that the billiard tables for all these models are the Weyl chambers of hyperbolic Kac-Moody algebras. It is striking that in $D=3$ spacetime dimensions, all coherently directed Dynkin diagrams of noncompact hyperbolic type and without closed circuits occur in the analysis of [dBS] and [HJ].

\medskip
\noindent The paper of [dBS] pointed out an error in the paper of Sa\c{c}lio\~{g}lu ([Sa])  who omitted 6 hyperbolic Dynkin diagrams in his account of the classification of symmetric and symmetrizable hyperbolic Dynkin diagrams. De Buyl and Schomblond gave the 6 Dynkin diagrams which they believe were omitted by Sa\c{c}lio\~{g}lu in [dBS], p 4491, though they did not verify symmetrizability. Some searching of the literature revealed other accounts of the classification by Li ([Li]) and Kobayashi and Morita ([KM]). However neither of these papers are as accessible or detailed as the paper of Sa\c{c}lio\~{g}lu. Moreover, there is no complete and accessible account of the classification of hyperbolic Dynkin diagrams in the literature that is free of errors. For example Kobayashi and Morita simply listed the symmetric or symmetrizable hyperbolic Dynkin diagrams in an appendix to their paper [KM] without verification and is not searchable in the literature. An account of Li's classification was also written by Wan in [W] but contains a number of omissions and misprints. There is also a wide variety of notation in use for given hyperbolic diagrams, with no systematic conventions.

\medskip
\noindent Comparing the independent classification results of Kac ([K]), Kobayashi-Morita ([KM]), Li ([Li]) and Sa\c{c}lio\~{g}lu ([Sa]), we note that the total number of  hyperbolic Dynkin diagrams in ranks 3-10 is 238, and is not in question. Moreover 142 of these are claimed to be symmetrizable by Kobayashi-Morita, though a criterion for symmetrizability was not given. Sa\c{c}lio\~{g}lu also did not give a reason why the 136 diagrams he demonstrated are symmetrizable. The results of Kobayashi-Morita appear to confirm the claim of de Buyl and Schomblond that Sa\c{c}lio\~{g}lu omitted 6 Dynkin diagrams, except for a missing verification of symmetrizability in [KM].

\medskip
\noindent While there are 4 independent accounts of different parts of the classification which can be compared and merged, we obtained our own account of the classification from scratch, using the well established methods described by the above authors. We generated the possible hyperbolic Dynkin diagrams independently of previous authors and we compared our results with the existing accounts. To ensure accuracy, this process was undertaken independently by at least 3 of the authors of this paper, and our results were compared. We confirm that the total number of  hyperbolic Dynkin diagrams in ranks 3-10 is 238. Our independent check that each of the 238 diagrams claimed by the above authors can occur has revealed further errata in the literature, which we list in Section 8.

\medskip
\noindent We believe that determining symmetrizability is a crucial component of the classification, and that the lack of criterion for verifying symmetrizability has led to errors in the past.  We obtain a criterion for symmetrizability which we apply to each hyperbolic Dynkin diagram. We can therefore confirm the claim that there are 142 symmetrizable hyperbolic Dynkin diagrams. Applying our symmetrizability criterion assisted us in obtaining a correct statement of the classification.

\medskip
\noindent Our symmetrizability criterion for Dynkin diagrams also leads to a simple proof  that the maximal rank of a Dynkin diagram of compact hyperbolic type is 5, while the maximal rank of a symmetrizable Dynkin diagram of compact hyperbolic type is 4 (Section 3).

\medskip
\noindent We give detailed and complete tables of all hyperbolic Dynkin diagrams: symmetric, symmetrizable and non symmetrizable, and we summarize the existing notation for hyperbolic diagrams (Section 7). For each symmetrizable hyperbolic generalized Cartan matrix, we give a symmetrization and hence the distinct root lengths in the corresponding root system.

\medskip
\noindent The Dynkin diagrams which are indefinite but not hyperbolic are too numerous to classify, however Gritsenko and Nikulin have a program to classify those indefinite Kac-Moody algebras whose root lattices coincide with Lorentzian lattices. The reader is referred to [GN] for a survey of this work which is based on earlier work of Vinberg ([Vi]) on the classification of Lorentzian lattices.

\medskip
\noindent  We also consider the properties of  the root systems of hyperbolic Kac-Moody algebras. It is natural to try to determine the structure of the Weyl group orbits on roots. Let $W$ denote the Weyl group of a Kac-Moody root system $\Delta$. In the finite dimensional case, all roots are Weyl group translates of simple roots. For infinite dimensional Kac-Moody root systems, there are additional mysterious roots of negative norm (`squared length') called imaginary roots. A root $\alpha\in\Delta$ is therefore called {\it real}  if there exists $w\in W$ such that $w(\alpha)$ is a
simple root. A root $\alpha$ which is not real is called {\it imaginary}. 

\medskip
\noindent  The real and imaginary roots of hyperbolic root systems and their images under the Weyl group are known to have a physical interpretation. For example, in cosmological billiards, the walls of the billiard table are related to physical fluxes that, in turn, are related to real roots. In [BGH]  Brown, Ganor and Helfgott show that real roots correspond to fluxes or instantons, and imaginary roots correspond to particles and branes. In [EHTW] Englert,  Houart,  Taormina, and West give a physical interpretation of Weyl group reflections in terms of M-theory.

\medskip
\noindent  In the finite dimensional case, the Weyl group is transitive on roots of the same norm, that is, roots of the same norm all lie in the same Weyl group orbit. The root systems of infinite dimensional algebras, such as hyperbolic algebras, have the mysterious property that roots of the same norm can lie in distinct Weyl group orbits. This was proven in [CCP] where the authors gave a simple criterion for checking if simple roots lie in the same orbit of the Weyl group, and this criterion can be checked easily on the Dynkin diagram. This criterion extends  naturally to all real roots of the Kac-Moody root system. We therefore include a complete tabulation of the Weyl group orbits on real roots of symmetrizable hyperbolic root systems which are included in our tables in Section 7.

\medskip
\noindent  We thank the referees for helpful comments which assisted us in clarifying some aspects of our exposition.


\section{Symmetrizability} \label{symm}

\medskip
\noindent Let $A=(a_{ij})$, $i,j\in\{1,2,\dots ,\ell\}$ be a {\it generalized Cartan matrix}. That is, the entries of $A$ are integers and the following conditions are satisfied ([K]):

\medskip
(1) $a_{ii} = 2$,

(2) $a_{ij} \leq 0$, $i\neq j$,

(3) $a_{ij} = 0$ implies $a_{ji} = 0$.

\begin{defn} A generalized Cartan matrix $A$ is called {\em indecomposable} if $A$ cannot be written as a block diagonal matrix, $A = diag(A_1, A_2)$ up to reordering of the rows and columns, where $A_1$, $A_2$ are generalized Cartan matrices.
\end {defn}

\medskip
\noindent Equivalently, a generalized Cartan matrix $A$ is indecomposable if and only if the corresponding Dynkin diagram $\mathcal{D}$ is connected.

\medskip
\noindent The generalized Cartan matrix $A$ is  {\it symmetrizable} if there exist nonzero rational numbers $d_1,\dots, d_\ell$, such that the matrix $DA$ is symmetric, where $D=diag(d_1,\dots, d_\ell)$.  We call $DA$ a {\it symmetrization} of $A$. A symmetrization  is unique up to a scalar multiple. Kac has given the following criterion for symmetrizability ([K], Exercise 2.1).

\begin{proposition} Let $A=(a_{ij})_{i,j=1,\dots, \ell}$ be a generalized Cartan matrix. Then $A$ is symmetrizable if and only if 
$$a_{i_1i_2}a_{i_2i_3}\dots a_{i_ki_1}\ =\ a_{i_2i_1}a_{i_3i_2}\dots a_{i_1i_k}$$
for each $i_1,\dots, i_k\in\{1,\dots ,\ell\}$, $k\leq \ell$, $i_s\neq i_{s+1}$ for $s$ mod $k$.
\end{proposition}

\medskip
\noindent {\bf Example.} Let $A=\left(\begin{matrix}
2 & -1 & -1  \\
-2 & 2 & -2 \\
-2 & -1 & 2 \\
\end{matrix}
\right)$. Then $a_{1 2}a_{2 3}a_{3 1}=-4\neq -2 = a_{2 1}a_{3 2}a_{1 3}$.  Hence, $A$ is not symmetrizable.

\medskip
\noindent We will give an equivalent criterion for determining symmetrizability based on the Dynkin diagram $\mathcal{D}$  of a generalized Cartan matrix $A$.  As a corollary we will
also be able to construct the symmetrizing diagonal matrix $D$, such that $DA$ is symmetric.  

\medskip
\noindent We recall the construction of the Dynkin diagram $\mathcal{D}$ from 
the generalized Cartan matrix $A$.  The vertices $V=\{v_1,\cdots, v_{\ell}\}$ correspond to the columns (or rows) of $A$ in that order.  The edge $e_{ij}=(v_i,v_j)$ ($i \neq j$)
between $v_i$ and $v_j$ depends on the entries $a_{ij}$, and $a_{ji}$ of $A$. The edges between $v_i$ and $v_j$ can be characterized as follows:

\medskip
\noindent  (a) No edge: if $a_{ij}=a_{ji}=0$;

\medskip
\noindent  (b) Single edge (symmetric):\hspace{3mm}
\xy
(2,2.5)*{v_i}; (11,2.5)*{v_j };
(2,0)*{\circ}; (11,0)*{\circ} **\dir{-};
\endxy
\ if $a_{ij} = a_{ji} = -1$;

\medskip
\noindent (c) Directed arrow with double, triple or quadruple edges (asymmetric): \xy
(2,2.5)*{v_i}; (11,2.5)*{v_j };
{\ar@2{->} (2,0)*{\circ};(11,0)*{\circ}}; 
\endxy ,
\xy
(2,2.5)*{v_i}; (11,2.5)*{v_j };
{\ar@3{->} (2,0)*{\circ};(11,0)*{\circ}}; 
\endxy,
\xy
(2,3)*{v_i}; (11,3)*{v_j };
(2.5,0)*{\circ}; (10.5,0)*{\circ};
(3.7,-0.25)*{}; (8.8,-0.25)*{} **\dir{-};
(3.7,0.25)*{}; (8.8,0.25)*{} **\dir{-};
(3.5,-0.8)*{}; (7.8,-0.8)*{} **\dir{-};
(3.5,0.8)*{}; (7.8,0.8)*{} **\dir{-};
(9.3,0)*{}; (7.3,1.2)*{} **\dir{-};
(9.3,0)*{}; (7.3,-1.2)*{} **\dir{-};
\endxy 
\ \ if
$a_{ij}=-1$, and $a_{ji}=-2, -3$, or $-4$, respectively;

\medskip
\noindent (d) Double edges with double-headed arrow (symmetric):\hspace{3mm}
\xy
(2,2.5)*{v_i}; (11,2.5)*{v_j };
{\ar@2{<->} (2,0)*{\circ};(11,0)*{\circ}}; 
\endxy
\ if $a_{ij}=a_{ji}=-2$;

\medskip
\noindent (e) Labeled edge (symmetric or asymmetric): 
\xy
(2,2.5)*{v_i}; (15,2.5)*{v_j };
(2,0)*{\circ}; (15,0)*{\circ} **\dir{-};
(6,2)*{a};(11,2)*{b}
\endxy
\ if it is none of the above types, and 
$a_{ij}=-a$, $a_{ji}=-b$, where $a,b \in \mathbb{Z}_{>0}$.

\medskip
\noindent We call an edge of a Dynkin diagram {\it symmetric} if $a_{ij} = a_{ji}$, and {\it asymmetric} otherwise.

\medskip
\noindent We will refer to an oriented edge $(v_i, v_j)$ of the type (c) as an arrow of multiplicity 2, 3, or 4 (respectively). 
The multiplicity of an oriented labeled edge of  type (e) is $b/a$.
 If the multiplicity of $(v_i,v_j)$ is $m$, then the multiplicity of the oppositely oriented edge, $(v_j,v_i)$, is $m^{-1}$.  All symmetric edges have multiplicity $1$.

\medskip
\noindent Note that the multiplicity of the oriented edge $(v_i,v_j)$ is given by $a_{ji}/a_{ij}$.  We denote this by 
$$ \displaystyle mult (v_i, v_j) = \frac{a_{ji}}{a_{ij}}. $$


\begin{defn}  A cycle $(v_{i_1},\cdots,v_{i_k},v_{i_1})$ in $\mathcal{D}$ is called balanced if the product of the multiplicities of the (oriented) edges in the cycle is $1$ when traversed in any particular direction (i.e., clockwise, or counter clockwise).
\end{defn}

\begin{theorem}\label{sym}  Let $\mathcal{D}$ be a Dynkin diagram.  Then $A$ is symmetrizable if and only if each cycle in $\mathcal{D}$ is balanced.
\end{theorem}

\noindent {\it Proof: } 
We will show that the above criterion based on the Dynkin diagram $D$ is equivalent to Kac's criterion, based on the generalized Cartan matrix $A$ for symmetrizability.

\medskip
\noindent Note that vertices $v_i$ and $v_j$ are connected by an edge in $\mathcal{D}$ if and only if $a_{ij} \neq 0$.
Let $i_1,\cdots,i_k \in \{1,\cdots,\ell\}$ be $k$ distinct indices.  The product
$a_{i_i i_2} a_{i_2 i_3} \cdots a_{i_k i_1}$ is nonzero if and only if the vertices $v_{i_1}, \cdots, v_{i_k}$ form a cycle in $\mathcal{D}$.  If the above product is zero then so is the `reverse' product 
$a_{i_2 i_1} \cdots a_{i_1 i_k}$ (since, we assume $a_{ij}=0 \Leftrightarrow a_{ji}=0$).  So the nontrivial conditions in Kac's criterion correspond to conditions on the cycles of $\mathcal{D}$.

\medskip
\noindent A nontrivial relation in Kac's criterion can be rewritten as
$$\displaystyle
\frac{a_{i_2 i_1}}{a_{i_1 i_2}} \cdot
\frac{a_{i_3 i_2}}{a_{i_2 i_3}} \cdot
\cdots \cdot
\frac{a_{i_1 i_k}}{a_{i_k i_1}} = 1.$$
The product in the left hand side of the above equation is the product of the multiplicities of the edges in the cycle $(v_{i_1}, \cdots, v_{i_k}, v_{i_1})$ in $\mathcal{D}$.
Hence, the result follows.
$\square$

\begin{proposition}\label{drel} Let $A$ be a symmetrizable generalized Cartan matrix, and $\mathcal{D}$ be the corresponding Dynkin diagram.  Let $D = diag (d_1, \cdots d_{\ell})$ be the diagonal matrix that symmetrizes $A$.  Then we have 
$$ \displaystyle
d_i = d_j \mathbf{\cdot} \, mult(v_i,v_j) = 
d_j \mathbf{\cdot} \, \frac{a_{ji}}{a_{ij}},
$$
whenever $v_i$ and $v_j$ are connected by an edge. Therefore, 
$mult(v_i,v_j) = a_{ji}/a_{ij} = d_i/d_j$.
\end{proposition}

\noindent
{\it Proof: }  For $i<j$, let $A_{ij}$ be the following submatrix of $A$ 
\[
\begin{pmatrix}
   a_{ii} & a_{ij} \\
   a_{ji} & a_{jj}
\end{pmatrix}
\]
Similarly, define $D_{ij} = diag (d_i, d_j)$.  If $DA$ is symmetric, then clearly $D_{ij}A_{ij}$ is symmetric.

\medskip
\noindent If the vertices $v_i$ and $v_j$ of $\mathcal{D}$ are connected by an edge, then the $a_{ij}$ and $a_{ji}$ are nonzero. Since $D_{ij}A_{ij}$ is symmetric we have
\[
\begin{pmatrix}
   d_i & 0 \\
   0 & d_j
\end{pmatrix}
\begin{pmatrix}
   a_{ii} & a_{ij} \\
   a_{ji} & a_{jj}
\end{pmatrix}
=
\begin{pmatrix}
  d_i a_{ii} & d_i a_{ij}\\
  d_j a_{ji} & d_j a_{jj}
\end{pmatrix}
\]
we have $d_i a_{ij} = d_j a_{ji}$.  
$\square$

\medskip\noindent{\bf Remark.} By Proposition \ref{drel}, given a Dynkin diagram $\mathcal{D}$ of a generalized Cartan matrix $A$, the symmetrizablity can be readily inferred, and the symmetrizing matrix $D$ (if $A$ is symmetrizable) can be readily computed by inspection of $\mathcal{D}$.

\begin{corollary}\label{acyclic} Any acyclic Dynkin diagram is symmetrizable.
\end{corollary}

\begin{corollary} Let $A$ be a $2\times 2$ generalized Cartan matrix. Then $A$ is symmetrizable.$\square$
\end{corollary}

\noindent
{\it Proof: }  The corresponding Dynkin diagram $\mathcal{D}$ has only two vertices.  So  $\mathcal{D}$ has no cycle.  Hence, by Corollary \ref{acyclic}, $A$ is symmetrizable. $\square$

\medskip\noindent{\bf Remark.} All finite and affine Dynkin diagrams are symmetric or symmetrizable. 

\medskip\noindent{\bf Remark.} It can be easily deduced that if $\mathcal{D}$ is a hyperbolic Dynkin diagram and if  $7\leq rank(\mathcal{D})\leq 10$ then $\mathcal{D}$ is symmetric or symmetrizable.

\begin{proposition} A Dynkin diagram is symmetric if and only if it contains only symmetric edges, that is, single edges, or double edges with  bi-directional arrows.\end{proposition}

\medskip
\noindent {\it Proof:} Let $A$ be the corresponding generalized Cartan matrix.  Then $A$ is symmetric if and only if $a_{ij} = a_{ji}$ for all $1 \leq i,j \leq \ell$, that is, when all edges $(v_i, v_j)$ are symmetric.
$\square$


\section{Classification} \label{class}

\medskip
\noindent The rank 2 hyperbolic generalized Cartan matrices, infinite in number are:
$$
A=\left( \begin{matrix}
2 & -a\\
-b & 2\\
\end{matrix}
\right) _{ab > 4}.
$$

\medskip
\noindent We recall that the following are the only 2$\times 2$ affine generalized Cartan matrices:
$$A_1^{(1)}=\left( \begin{matrix}
2 & -2\\
-2 & 2
\end{matrix}
\right),\ A_2^{(2)}=\left( \begin{matrix}
2 & -1\\
-4 & 2
\end{matrix}
\right).$$

\begin{proposition} A symmetrizable hyperbolic generalized Cartan matrix contains an $A_1^{(1)}$ or $A_2^{(2)}$ proper indecomposable  submatrix if and only if rank $A=3$ and $A$ has  non-compact type.
\end{proposition} 
\noindent {\it{Proof:}} A symmetrizable hyperbolic generalized Cartan matrix cannot contain an $A_1^{(1)}$ or $A_2^{(2)}$ indecomposable proper submatrix if rank $A=\ell>3$, since the Dynkin diagram corresponding to $A$ has $\ell$ vertices, and the 3 vertex connected subdiagram consisting of  $A_1^{(1)}$ or $A_2^{(2)}$  plus an additional vertex would then be neither affine nor finite. Thus if $A$ is a symmetrizable hyperbolic generalized Cartan matrix with an $A_1^{(1)}$ or $A_2^{(2)}$ indecomposable submatrix, then the rank of $A$ is  3, and $A$ has  non-compact type .

\medskip
\noindent Conversely, every hyperbolic diagram of rank 3 of non-compact type must contain an $A_1^{(1)}$ or $A_2^{(2)}$ indecomposable subdiagram, since it must contain a subdiagram of affine type with 2 vertices. $\square$

\begin{corollary} Let $A$ be a rank 3 symmetrizable hyperbolic generalized Cartan matrix of non-compact type. The following  conditions are equivalent.
\begin{enumerate}
\item[(a)] The Dynkin diagram for $A$ has an $A_1^{(1)}$ or $A_2^{(2)}$ proper connected subdiagram.
\item[(b)]  $A$ has a proper indecomposable affine submatrix $B=(b_{ij})$ such that for some $i,j\in\{1,2,3\}$, $i\neq j$, $b_{ij}b_{ji}= 4$.

\end{enumerate}
\end{corollary}

\begin{corollary} Let $A_0$ be a rank 2 proper submatrix of a symmetrizable indecomposable hyperbolic generalized Cartan matrix $A$ of rank $\geq 4$. Then $A_0$ is of finite type.
\end{corollary} 
\noindent {\it{Proof:}} Since the rank of $A$ is $\geq 4$, by the proposition, $A$ cannot contain a $A_1^{(1)}$ or $A_2^{(2)}$ indecomposable proper submatrix. Let $A_0$ be any proper rank 2 submatrix. Then $A_0$ itself cannot be affine or hyperbolic. The diagonal entries of $A_0$ must equal 2 since $A_0$ is proper. The remaining entries must necessarily be 0, -1, -2 or -3, with det$(A_0)>0$. That is, $A_0$ is of finite type. $\square$

\begin{corollary} Let $\mathcal{D}$ be a hyperbolic Dynkin diagram with $n$ vertices. Then any proper connected affine subdiagram has $n-1$ vertices.
\end{corollary}
\noindent {\it{Proof:}} Suppose conversely that $\mathcal{D}$ has a proper connected affine subdiagram $\mathcal{D}_0$ with  $n-s$ vertices, $1<s<n$. Then $\mathcal{D}_0$ plus an additional vertex would be a proper connected subdiagram of $\mathcal{D}$  of hyperbolic type, which is a contradiction. $\square$

\begin{lemma} Let $\mathcal{D}$ be a hyperbolic Dynkin diagram of rank
$n$. Then, there exists
a connected subdiagram with $n-1$ vertices.
\end{lemma}

\begin{proof} Suppose $\mathcal{D}$ is of compact hyperbolic type. Then, if $\mathcal{D}$
contains a cycle, it must itself be a cycle;
otherwise, deleting  off-cycle vertices will give a cycle as a proper
subdiagram. Since there are no cycles of finite type, this would give a subdiagram of affine type which is a contradiction. If $\mathcal{D}$ is a cycle,
then deleting a vertex will leave the remaining $n-1$ vertices connected,
because they all lie on the same cycle.
If $\mathcal{D}$ does not contain a cycle, then deleting any vertex that
is connected to only $1$ other vertex
will leave a connected subdiagram of $n-1$ vertices, because all the other
vertices were connected in the original
diagram.

\medskip
\noindent Now suppose $\mathcal{D}$ is non-compact. Then it contains a
proper connected affine diagram,
and so, by Corollary 3.4, we have that this affine subdiagram must have
$n-1$ vertices. Therefore, all hyperbolic
Dynkin diagrams of rank $n$ have a connected subdiagram of rank $n-1$.
\end{proof}

\begin{proposition} Let $\mathcal{D}$ be a hyperbolic Dynkin diagram. Then $rank(\mathcal{D})\leq 10$. 
\end{proposition}
\medskip
\noindent We refer the reader to ([K], Ch 4), [Sa] and [Li] for a proof.

\medskip
\begin{proposition} Let $\mathcal{D}$ be a hyperbolic Dynkin diagram. 

\medskip
\noindent (1) If $\mathcal{D}$ has compact hyperbolic type, then $rank(\mathcal{D})\leq 5$. 

\medskip
\noindent (2) If $\mathcal{D}$ is symmetric or symmetrizable and has compact hyperbolic type, then $rank(\mathcal{D})\leq 4$.
\end{proposition}
\medskip
\noindent {\it Proof:} Let $\ell$ denote the maximal rank of any Dynkin diagram of compact hyperbolic type. Then $\ell$ is at least 4, since we have 
the following symmetrizable diagram of compact hyperbolic type (no. 136 in Section 7):

\medskip

\begin{figure}[h!]
\resizebox{.55 in}{.51 in}{\includegraphics*[viewport=140 680 182 720 ]{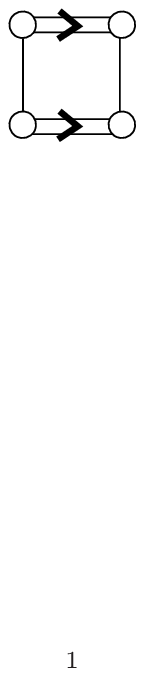}}
\end{figure}

\medskip
\noindent which has $B_3$, $C_3$, $B_2$, $A_2$, and $A_1$ as its proper connected subdiagrams.

\medskip
\noindent Let $\mathcal{D}$ be a Dynkin diagram of compact hyperbolic type and rank $>4$. First suppose $\mathcal{D}$ contains a cycle. Then $\mathcal{D}$ must itself be a cycle; otherwise, deleting  off-cycle vertices will give a cycle as a proper subdiagram. Since there are no cycles of finite type, this would give a subdiagram of affine type which is a contradiction. Moreover $\mathcal{D}$ cannot contain all single edges, since this would mean $\mathcal{D}$ is affine of type $A^{(1)}_{\ell}$. Also, the only multiple edges $\mathcal{D}$ contains are double edges with single arrows; otherwise, since $\mathcal{D}$ has more than four vertices  there is a three-vertex  subdiagram which is not of finite type. (Note that $G_2$ is the only finite diagram containing a multiple edge other than a double edge with a single arrow and has only two vertices). Similarly, $\mathcal{D}$ contains at most one multiple edge: otherwise, since $\mathcal{D}$ has at least five vertices, we can delete one and obtain a diagram with two multiple edges, which cannot be of finite type. Thus the only cyclic Dynkin diagrams of compact hyperbolic type contain a unique double edge:

\medskip

\begin{figure}[h!]
\resizebox{.67 in}{.64 in}{\includegraphics*[viewport=137 670 188 719 ]{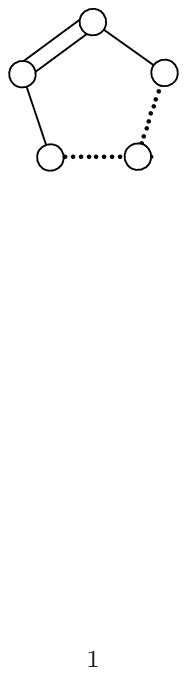}}
\end{figure}

\medskip

\noindent Such a diagram contains the diagram

\medskip
\textcolor{white}. 

\begin{figure}[h!]
\resizebox{1.74 in}{.16 in}{\includegraphics*[viewport=200 702 340 715 ]{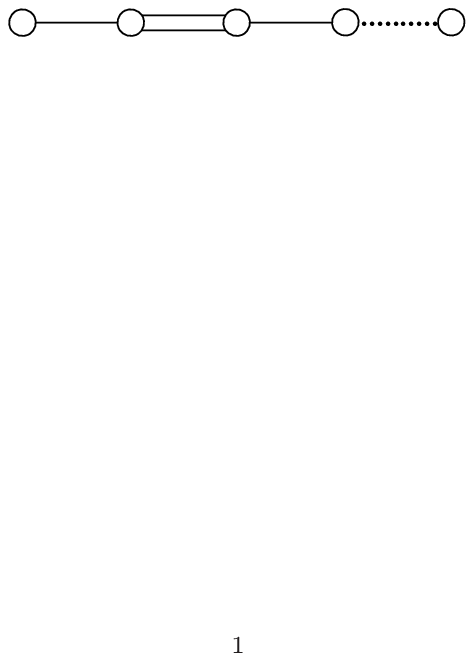}}
\end{figure}

\medskip
\noindent as a proper subdiagram, hence is of compact type only when it contains five vertices. Thus if $\mathcal{D}$ contains a cycle, $\mathcal{D}$ must be diagram no. 183 in Section 7:
      
\vspace{18.25 pt}
\begin{figure}[h!]
\resizebox{.64 in}{.64 in}{\includegraphics*[viewport=137 669 187 719 ]{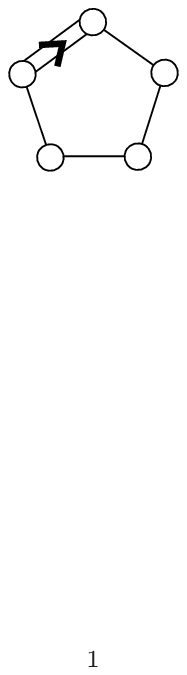}}
\end{figure}

\medskip
\noindent Now suppose $\mathcal{D}$ is a tree. As in the previous case, the only multiple edges $\mathcal{D}$ can contain are double edges with  single arrows. As in the previous case, $\mathcal{D}$ may contain no more than one multiple edge; otherwise $\mathcal{D}$ is either affine of the form:

\medskip

\begin{figure}[h!]
\resizebox{1.36 in}{.18 in}{\includegraphics*[viewport=215 704 323 718 ]{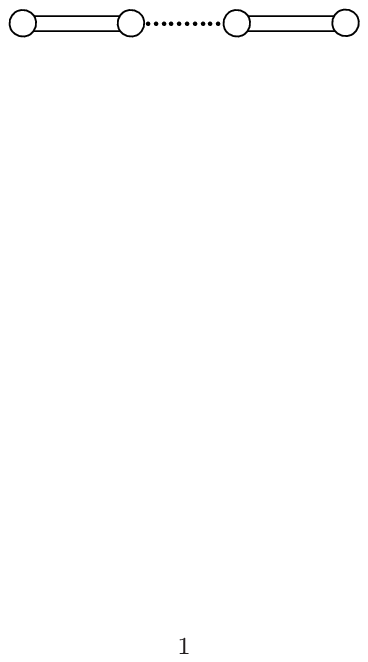}}
\end{figure}

\medskip
\noindent or contains such a diagram as a proper subdiagram. Also, $\mathcal{D}$ cannot contain  both a double edge and a branch point (that is, a vertex connected to at least three other vertices), since then $\mathcal{D}$ would either be affine of the form:

\medskip

\begin{figure}[h!]
\resizebox{1.32 in}{.36 in}{\includegraphics*[viewport=112 692 211 719]{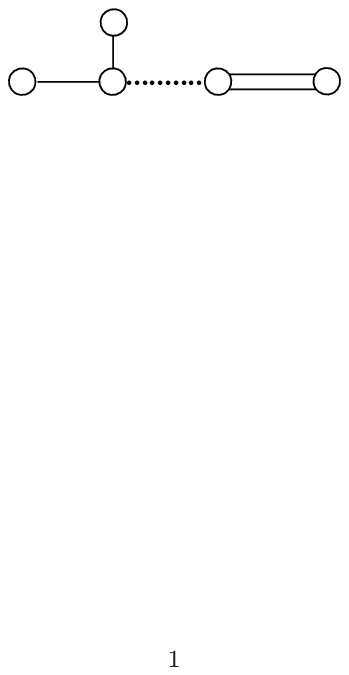}}
\end{figure}

\medskip
\noindent or would contain such a diagram as a proper subdiagram. Likewise, $\mathcal{D}$ cannot contain two branch points, since then it would contain the affine diagram:

\medskip

\begin{figure}[h!]
\resizebox{1.37 in}{.36 in}{\includegraphics*[viewport=254 692 357 719 ]{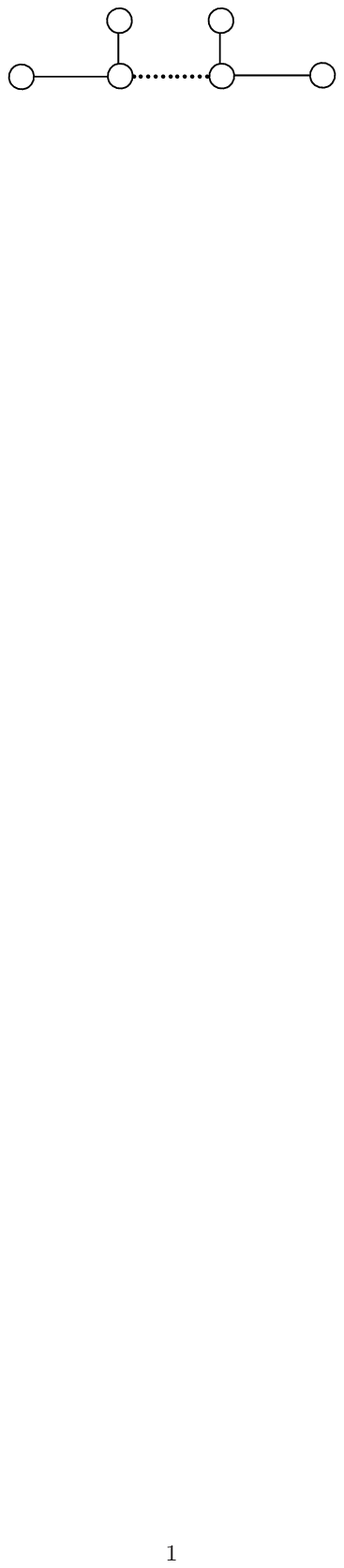}}
\end{figure}

\medskip
\noindent Note that this includes the case of a single vertex connected to four or more other vertices.

\medskip
\noindent Since a linear chain containing all single edges is the finite type diagram $A_n$, the two remaining possibilities are 

\medskip
\noindent (i) $\mathcal{D}$ is a linear chain containing one double edge, and 

\medskip
\noindent (ii) $\mathcal{D}$ contains one branch point, with three vertices attached,  and all single edges.

\medskip

\begin{figure}[h!]
\resizebox{3.62 in}{1.06 in}{\includegraphics*[viewport=168 632 451 715 ]{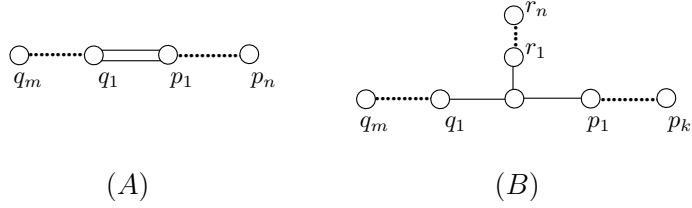}}
\caption{In diagram $(A)$, $\mathcal{D}$ is a linear chain containing one double edge.  In diagram $(B)$, $\mathcal{D}$ contains one branch point with three vertices attached, and all single edges.}
\end{figure}

\medskip

\noindent We now refer to the notation of Figure 1. For case (i), we may without loss of generality assume that $m\leq n$. If $m=1$, then $\mathcal{D}$ is $B_{n+1}$ or $C_{n+1}$, so $m\neq 1$. If $m=2$, then $n\geq 3$ since $\mathcal{D}$ contains at least five vertices; then $\mathcal{D}$ either is or contains as a proper subdiagram $E_6^{(2)}$ or $F_4^{(1)}$, so $m\neq 2$. If $m\geq 3$, then $\mathcal{D}$ contains $E_6^{(2)}$ or $F_4^{(1)}$ as a proper subdiagram, so case (i) cannot occur.

\medskip
\noindent Now we eliminate case (ii). We may assume $k\leq m\leq n$. First suppose $k=1$. If $m=1$, then the diagram is $D_{n+3}$. If $m=2$, then $\mathcal{D}$ is either $E_6$, $E_7$, $E_8$, $E_8^{(1)}$, or contains $E_8^{(1)}$ as a proper subdiagram. If $m\geq 3$, then $\mathcal{D}$ is either $E_7^{(1)}$ or contains $E_7^{(1)}$ as a proper subdiagram. This eliminates the possibility that $k=1$. If $k\geq 2$, then $\mathcal{D}$ is either $E_6^{(1)}$ or contains $E_6^{(1)}$ as a proper subdiagram. Thus case (ii) cannot occur.

\medskip
\noindent This completes the proof that there are no acyclic compact hyperbolic diagrams of rank greater than four. Since the only compact hyperbolic diagram of rank greater than 4 is the non-symmetrizable rank $5$ cycle with a unique double edge (diagram 183 in Table 16, Section 7), the result follows. $\square$

\section{Root lengths} \label{rootlengths}

\medskip
\noindent Let $\rho$ denote the number of distinct root lengths of real roots 
in a root system corresponding to a Dynkin diagram $\mathcal{D}$. It is observed in 
([K], \S5.1) that if $m$ is the maximum number of arrows in a coherently directed path 
in $\mathcal{D}$, then there are simple roots of $m+1$ distinct root lengths. In fact,
we have that for simple roots $\alpha_{i},\alpha_{j}$, 
$$\frac{d_{j}}{d_{i}}=\frac{(\alpha_{j}|\alpha_{j})}{(\alpha_{i}|\alpha_{i})},$$
where $D=diag(d_{1},...,d_{\ell})$ gives a symmetrization of the generalized Cartan 
matrix $A$ associated to $\mathcal{D}$. Thus, the number of simple roots 
with distinct root length equals the number of distinct $d_{i}$ in $D$. 

\medskip\noindent A real root is the image of a simple root under the action of an element of the Weyl group. Note that the 
elements of the Weyl group preserve root length. Therefore, the number of distinct lengths of
real roots in a root system equals the number of distinct lengths of the simple roots of the root system, and 
so  $\rho$ equals the number of distinct $d_{i}$ in $D$.

\medskip\noindent It is well known that for finite root systems, $\rho$ is at most $2$, and for affine root systems, $\rho$ 
is at most $3$ ([K],\S5.1). We recall that all finite and affine Dynkin diagrams are symmetrizable. Hence, 
in a discussion about root lengths for such root systems, the assumption of symmetrizability is not necessary.  

\begin{proposition} Let $\rho$ denote the number of distinct root lengths
of real roots in a root system corresponding to a Dynkin diagram $\mathcal{D}$. 
If $\mathcal{D}$ is  symmetrizable and hyperbolic then $\rho$ is at most $4$.
\end{proposition}
 
\begin{proof} Let $\mathcal{D}$ be a hyperbolic Dynkin diagram of rank $n$, and let
$\mathcal{D}'$ be any connected subdiagram with $n-1$ vertices. Since $\mathcal{D}$ 
is hyperbolic, $\mathcal{D}'$ must be an affine or finite Dynkin diagram. Therefore, the
simple roots represented in $\mathcal{D}'$ have at most $3$ root lengths. The additional
vertex in $\mathcal{D}$ corresponds to a simple root that may have a different root
length from all the roots in $\mathcal{D}'$. Thus, $\mathcal{D}$ has at most $4$ distinct 
root lengths. \end{proof}

\medskip
\noindent  The only  affine root systems with the maximal number of real root lengths are $A_{2\ell}^{(2)}$, $\ell\geq 2$ ([K]). The only  hyperbolic root system with the maximal number of real root lengths is no.  173  in the tables of Section 7.

\section{Weyl group orbits on real roots}

\medskip
\noindent Let $A$ be a  symmetrizable generalized Cartan matrix. Let $\mathcal{D}=\mathcal{D}(A)$ be the corresponding Dynkin diagram with vertices indexed by $I=\{1,2,\ldots,\ell \}$. We say that $\mathcal{D}$ is {\it simply laced} if  $\mathcal{D}$ contains no multiple edges. We let $\mathcal{D}_{\ast}$ denote the graph obtained from $\mathcal{D}$ by deleting all multiple edges, including arrows and edge labels. Let $\mathcal{D}_1,\dots,\mathcal{D}_s$ denote the connected subdiagrams of $\mathcal{D}_{\ast}$. Then each $\mathcal{D}_i$ is simply laced. We call $\mathcal{D}_{\ast}$ the {\it simply laced skeleton} of $\mathcal{D}$. We may describe the graph $\mathcal{D}_{\ast}$ as follows

\medskip
$Vertices(\mathcal{D}_{\ast})=Vertices(\mathcal{D})$ with the same labelling, that is, indexed by $I=\{1,2,\ldots,\ell \}$

$Edges(\mathcal{D}_{\ast})=\bigcup_{i=1}^s \ Edges(\mathcal{D}_i)$

\medskip
\noindent  Vertices $v_i$ and $v_j$ are adjacent in $\mathcal{D}_{\ast}$ if and only if $v_i$ and $v_j$  are connected in $\mathcal{D}$ by a single edge with no arrows or edge labels which occurs if and only if  $a_{ij}=a_{ji}=-1$. 

\medskip
\noindent If $\mathcal{D}$ is simply laced  then $\mathcal{D}_{\ast}=\mathcal{D}$. Note that $\mathcal{D}_{\ast}$ may not be connected.

\medskip
\noindent   In [CCP], the authors prove the following.
\begin{theorem} Let $\mathcal{D}$ be a Dynkin diagram, let $\Pi=\{\alpha_1,\dots ,\alpha_{\ell}\}$ be the simple roots of the corresponding root system and let $W$ denote the Weyl group. Let  $\mathcal{D}^J_{\ast}$ denote a connected subdiagram of $\mathcal{D}_{\ast}$ whose vertices are indexed by $J\subseteq I$. Then 

\medskip
\noindent (1) If $j\in J$ and $k \notin  J$, we have ${W}\alpha_j\cap {W}\alpha_k=\varnothing$.

\medskip
\noindent (2) For all $i,j\in J$ there exists $w\in {W}$ such that $\alpha_i=w\alpha_j$.

\medskip
\noindent (3) For all $i,j\in J$, ${W}\alpha_i={W}\alpha_j$.

\end{theorem}

\medskip
\noindent The following corollary is immediate.
\begin{corollary} \label{maincor}  Let $\mathcal{D}$ be a Dynkin diagram, let $\Pi=\{\alpha_1,\dots ,\alpha_{\ell} \}$ be the simple roots of the corresponding root system and let $W$ denote the Weyl group. Then for $i\neq j$, the simple roots $\alpha_i$ and $\alpha_j$ are in the same $W$-orbit if and only if vertices $v_i$ and $v_j$ in the Dynkin diagram corresponding to $\alpha_i$ and $\alpha_j$ are connected by a path of single edgess in $\mathcal{D}$.

\end{corollary} 

\noindent
If a single $W$-orbit contains $m$ simple roots $\alpha_{i_1}, \alpha_{i_2}, \ldots, \alpha_{i_m}$, this orbit is written as $$W\{\alpha_{i_1}, \alpha_{i_2}, \ldots, \alpha_{i_m}\}.$$

\begin{corollary} \label{realroots} Let $\mathcal{D}$ be a Dynkin diagram, let $\Pi=\{\alpha_1,\dots ,\alpha_{\ell} \}$ be the simple roots of the corresponding root system and let $W$ denote the Weyl group. Let  $\mathcal{D}^{J_1}_{\ast}$, $\mathcal{D}^{J_2}_{\ast}$, $\dots$,  $\mathcal{D}^{J_t}_{\ast}$ denote the connected subdiagrams of $\mathcal{D}_{\ast}$ whose vertices are indexed by $J_1,J_2,\dots ,J_t\subseteq I$. Then the set of real roots $\Phi=W\Pi$ is
$$\Phi=W\{\alpha_{J_1}\}\sqcup W\{\alpha_{J_2}\}\sqcup\dots \sqcup W\{\alpha_{J_t}\},$$
where $\alpha_{J_s}$ denotes the subset of simple roots indexed by the subset $J_s \subseteq I$.
\end{corollary}

\medskip
\noindent Given any Dynkin diagram $\mathcal{D}$, we can therefore determine the disjoint orbits of the Weyl group on real roots by determining the simply laced skeleton $\mathcal{D}_{\ast}$. We tabulate the disjoint orbits  for each hyperbolic Dynkin diagram $\mathcal{D}$ in Section 7. 

\section{Extended and overextended Dynkin diagrams} \label{extended}

Let $\Delta$ be a finite root system, that is, the Dynkin diagram of a root system of a finite dimensional Lie algebra. We assume that $\Delta$ is indecomposable. In this case there is no decomposition of $\Delta$ into a union of 2 subsets where every root in one subset is orthogonal to every root in the other. Let $\Pi=\{\alpha_1,\dots ,\alpha_{\ell}\}$ be the simple roots of $\Delta$. For $\Delta$ indecomposable, there is a unique root $\delta$ called the {\it maximal root} that is a linear combination of the simple roots with positive integer coefficients. The maximal root $\delta$  satisfies $(\delta,\alpha)\geq 0$ for every simple root $\alpha$ and $(\delta,\beta)> 0$ for some simple root $\beta$, where $(\cdot\mid\cdot)$ is the positive definite symmetric bilinear form corresponding to $\Delta$ ([OV]).

\medskip
\noindent  Let $\alpha_0=-\delta$ and let $\Pi'=\Pi\cup\{\alpha_0\}$. Then $\Pi'$ is called the {\it extended system of simple roots} corresponding to $\Pi$. The Dynkin diagram of $\Pi'$ is called the {\it extended Dynkin diagram} corresponding to $\Delta$. An extended Dynkin diagram has a vertex labeled 0 corresponding to the root $\alpha_0$. All untwisted affine Dynkin diagrams  are extended Dynkin diagrams ([OV]).

\medskip
\noindent A generalized Cartan matrix $A$ is called {\it Lorentzian} if $det(A)\neq 0$ and $A$ has  exactly one negative eigenvalue. A  {\it Lorentzian Dynkin diagram} is the Dynkin diagram of a  Lorentzian generalized Cartan matrix. The notion of an extended Dynkin diagram first appeared in the classification of semisimple algebraic groups (see for example ([Ti])). I. Frenkel was the first to describe certain Lorentzian Dynkin diagrams obtained as further extensions of extended Dynkin diagrams ([F]).  A {\it Lorentzian extension} $\mathcal{D}$ of an untwisted affine Dynkin diagram $\mathcal{D}_0$ is a Dynkin diagram obtained by adding one vertex, labeled $-1$, to $\mathcal{D}_0$ and connecting the vertex $-1$ to the vertex of $\mathcal{D}_0$ labeled 0 with a single edge.

\medskip
\noindent Every Lorentzian extension of an untwisted affine Dynkin diagram is a Lorentzian Dynkin diagram, in fact a hyperbolic Dynkin diagram. There are also Lorentzian extensions of twisted affine Dynkin diagrams (see for example, ([HPS], 4.9.3)). Lorentzian extensions of affine Dynkin diagrams {\it Aff} are denoted {\it Aff} $^{\wedge}$ in our tables.

\medskip
\noindent {\bf Example - $E_{10}$:} Let $\Delta$ be the Dynkin diagram for $E_8$. We label the first vertex of the `long tail' by 1. Adding a vertex labeled 0 and connecting vertices 0 and 1 by a single edge yields the extended Dynkin diagram $\Delta'$ which corresponds to the affine Kac-Moody algebra $E_9=E_8^{(1)}$. Adding a further vertex labeled $-1$ and connecting vertices $-1$ and 0 by a single edge yields the overextended Dynkin diagram  which corresponds to the hyperbolic Kac-Moody algebra $E_{10}=E_8^{{(1)\wedge}}$. 

\medskip
\begin{figure}[h!]
\resizebox{2.23 in}{.5 in}{\includegraphics*[viewport=204 679 405 722 ]{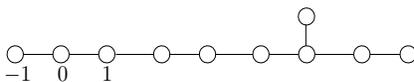}}
\caption{The Dynkin diagram for $E_{10}$.}
\end{figure}

\section{Notation and tables} \label{tables}

\medskip
\noindent We present below a comprehensive tabulation of all hyperbolic diagrams listed in [Sa], [KM] and [Li] with all errata from these tables corrected.  In particular we include the missing diagrams of [Sa] that were pointed out by [dBS]. We present a summary of the errata in the existing literature in Section 8.

\medskip
\noindent The diagrams in our tables generally follow Li's ordering and orientation of edges. Symmetrizable diagrams follow Sa\c{c}lio\~{g}lu's labeling. In these cases, the orientation of edges used by Li is changed if necessary. 

\medskip
\noindent The Dynkin diagrams in our tables correspond to {\it isomorphism classes} of Kac-Moody algebras. Kac and Peterson have shown that  Kac-Moody algebras over the same field are isomorphic if and only if their generalized Cartan matrices can be obtained from each other by reordering the index set ([KP]). This induces an automorphism of a Dynkin diagram and hence does not change the isomorphism class of the Kac-Moody algebra.

\medskip
\noindent Given a Dynkin diagram $\mathcal{D}$, the {\it dual Dynkin diagram} $\mathcal{D}^{dual}$ is obtained from $\mathcal{D}$ by  changing the directions of the arrows. This corresponds to taking the reciprocal of the ratio of the corresponding root lengths. If a Dynkin diagram $\mathcal{D}$ is self-dual, that is, $\mathcal{D}=\mathcal{D}^{dual}$, then the Kac-Moody algebras corresponding to $\mathcal{D}$ and $\mathcal{D}^{dual}$ are isomorphic. Dual diagrams which give rise to a different isomorphism class of Kac-Moody algebras appear explicitly. If $\mathcal{D}\neq \mathcal{D}^{dual}$ then $\mathcal{D}^{dual}$ immediately follows $\mathcal{D}$ in the tables.  In these cases, we depart from the ordering used by Li.

\medskip
\noindent The first column of the table represents an enumeration index for the diagrams. The second column `Other Notation' lists the common notations in use in the literature for the diagram $\mathcal{D}$. The first in the list is Li's notation which is of the form $H^{(rank(\mathcal{D}))}_n$,  where $n$ is an enumeration index. The diagrams are arranged by rank.

\medskip
\noindent  In rank 3 where relevant, we also use the notation $I\mathfrak{g}(a,b)$ which corresponds to the generalized Cartan matrix
$$\left(\begin{matrix}
{2} & {-b} & {0} & {\dots} & 0  \\
{-a} & {} & {} & {} & {} \\
{0} & {} & {C(\mathfrak{g})} & {} & {}  \\
{\vdots} & {} & {} & {} & {} \\
0 & {} & {}  & {} & {} \\
\end{matrix}
\right)$$
where $C(\mathfrak{g})$ is the Cartan matrix of a  Lie algebra or Kac-Moody algebra $\mathfrak{g}$. For many positive integer values of $a$ and $b$, $I\mathfrak{g}(a,b)$ is a generalized Cartan matrix of indefinite type. We use the standard finite or affine notation for $C(\mathfrak{g})$, and we refer the reader to [K] for tables and notation of   finite or affine type. Built into the notation $I\mathfrak{g}(a,b)$ is an assumption that the first vertex of the Dynkin diagram is connected to the second vertex (ordered left to right) and not to any other vertex. Thus in many cases we do not list the $I\mathfrak{g}(a,b)$ notation since it corresponds to a different ordering than our chosen ordering of vertices.

\medskip
\noindent  For Dynkin diagrams of noncompact hyperbolic type, the third index represents Kac's notation $AE_n$, $BE_n$, $CE_n$, $DE_n$, $T(p,q,r)$ where appropriate. Where relevant, the fourth index is of the form {\it Aff} $^{\wedge}$, where  `{\it Aff}' is a Dynkin diagram of affine type and ` ${}^{\wedge}$ ' represents an extension of {\it Aff} by adding one vertex and a single edge.  As in the previous section, we follow the convention of extending at the vertex labelled 0, though this labeling does not appear in the tables. The notation $H${\it Aff} is used in some papers in place of {\it Aff} $^{\wedge}$, and the notation {\it Aff} $'{}^{\wedge}$ is used to denote the dual of {\it Aff} $^{\wedge}$.

\medskip
\noindent The third column contains the Dynkin diagrams $\mathcal{D}$. If $\mathcal{D}$ is symmetric or symmetrizable, we use Sa\c{c}lio\~{g}lu's labelling of the vertices. The fourth column contains the diagonal matrix $diag(d_1,\dots , d_{\ell})$  in the case that $\mathcal{D}$ is symmetric or symmetrizable.  In this case the diagonal matrix $diag(d_1,\dots , d_{\ell})$ gives a symmetrization of the generalized Cartan matrix $A$ corresponding to $\mathcal{D}$. The $d_i$ have been rescaled to take integer values. If all $d_i$ equal 1, then 
$\mathcal{D}$ is symmetric. The maximal number of possible root lengths in the  root system corresponding to $\mathcal{D}$ is  the number of distinct $d_i$.  In the case that $\mathcal{D}$ is not symmetric or symmetrizable, the fourth column entry is `N.S.' for `not symmetrizable'.

\medskip
\noindent The fifth column contains the disjoint Weyl group orbits corresponding to the indexing of vertices in the Dynkin diagram.  If the Dynkin diagram is not symmetrizable, the fifth column entry is left blank.

\newpage

\begin{table}[h!]
  \caption{Rank 3 compact diagrams}
  \centering
    \resizebox{6.5in}{6.5 in }{\includegraphics*[viewport=35 35 720 720 ]{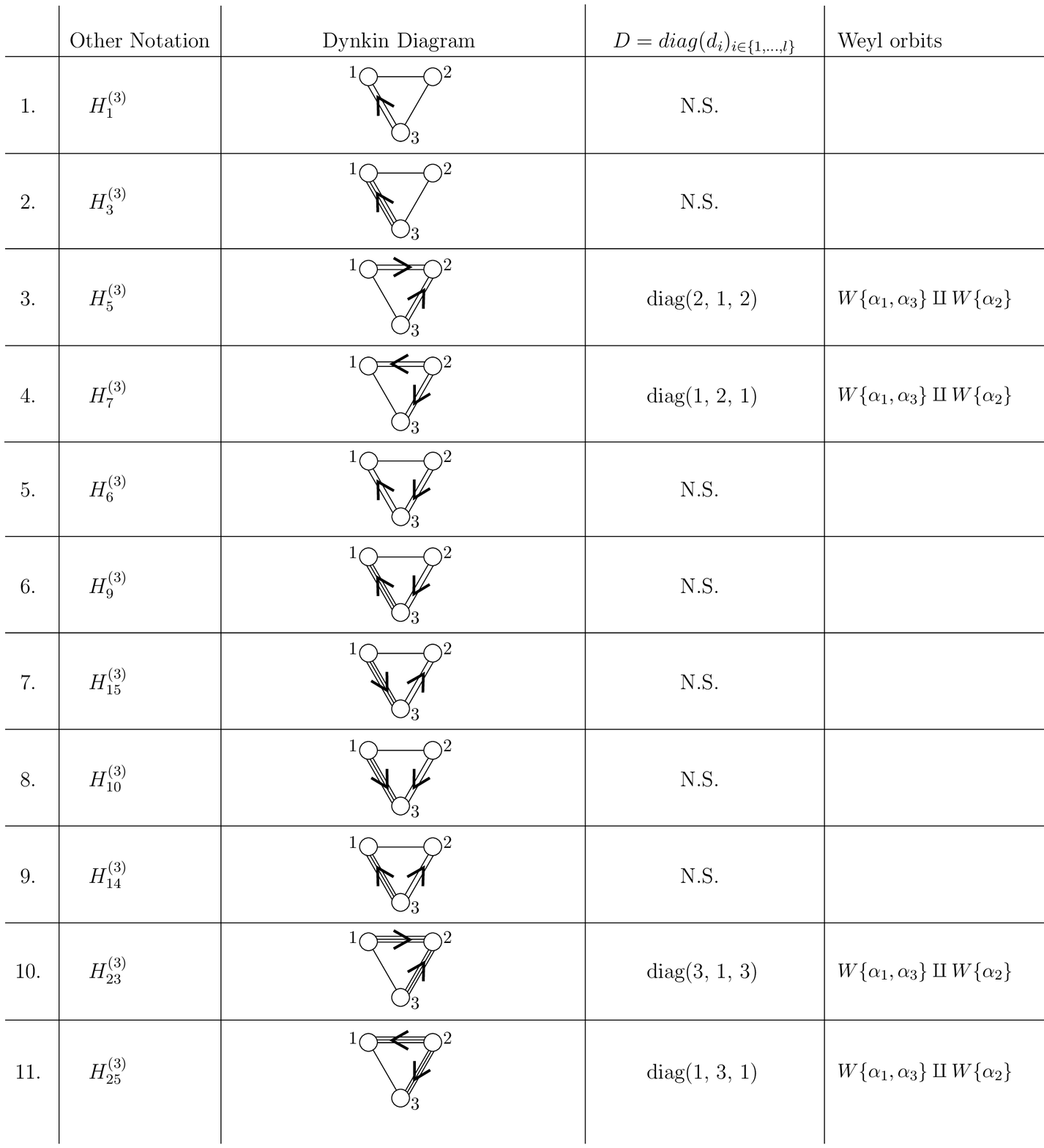}}
\end{table}

\newpage

\begin{table}[h!]
  \caption{Rank 3 compact diagrams (continued)}
  \centering
\resizebox{6.5in}{6.5 in }{\includegraphics*[viewport=35 35 720 720 ]{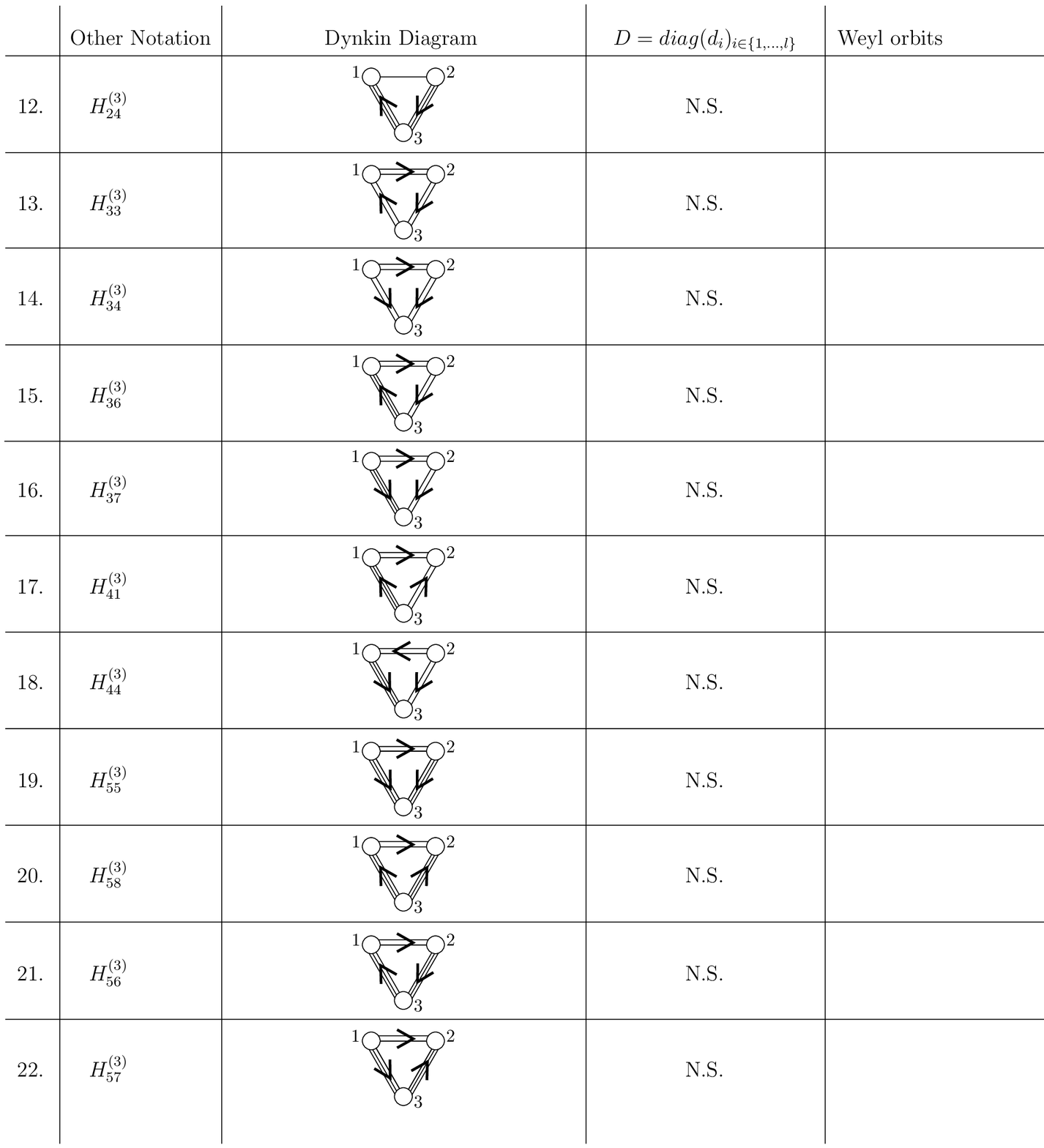}}
\end{table}

\newpage

\begin{table}[h!]
  \caption{Rank 3 compact diagrams (continued)}
  \centering
\resizebox{6.5in}{6.5 in }{\includegraphics*[viewport=35 35 720 720 ]{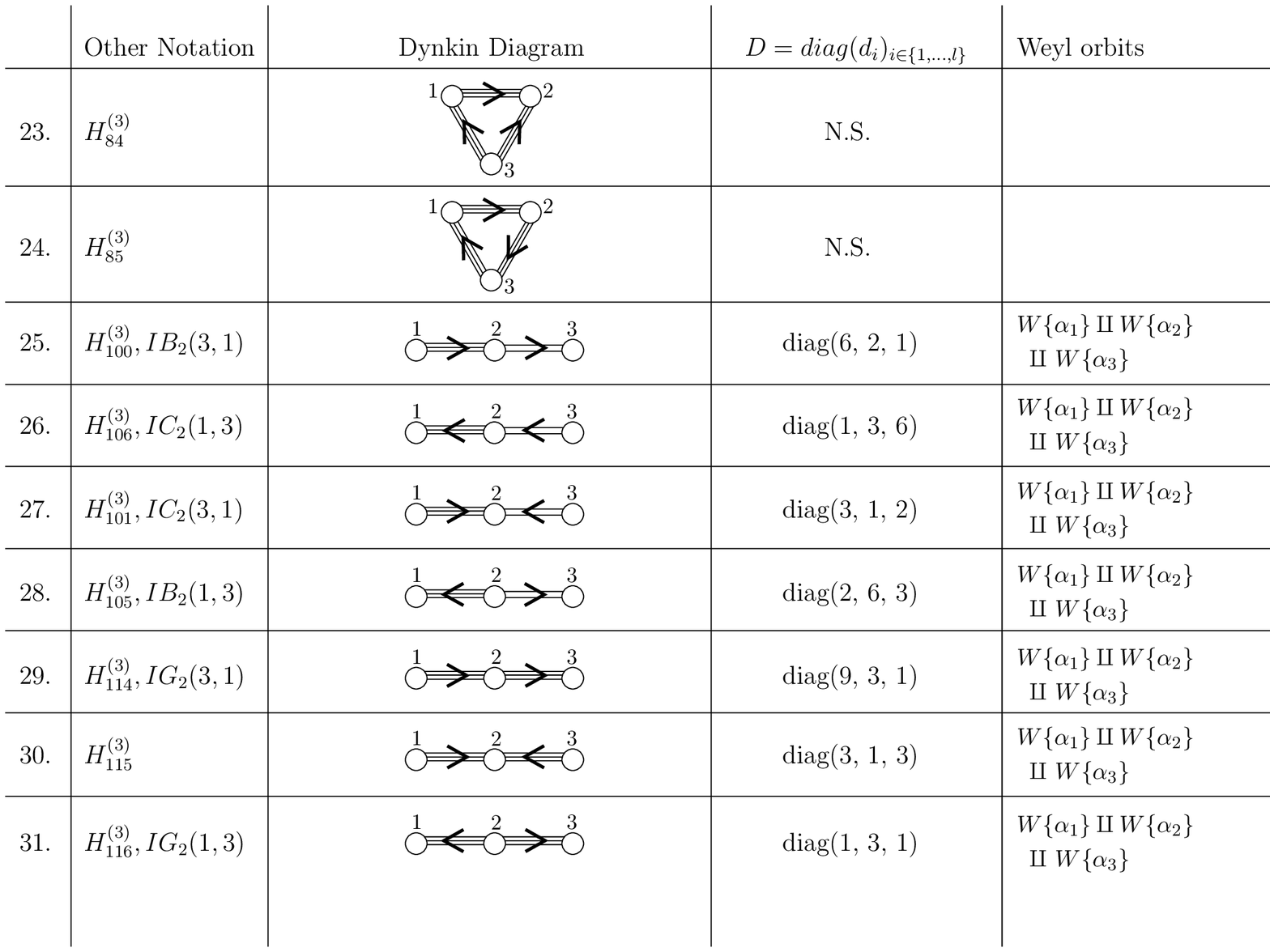}}
\end{table}

\newpage

\begin{table}[h!]
  \caption{Rank 3 non-compact diagrams}
  \centering
\resizebox{6.5in}{6.5 in }{\includegraphics*[viewport=35 35 720 720 ]{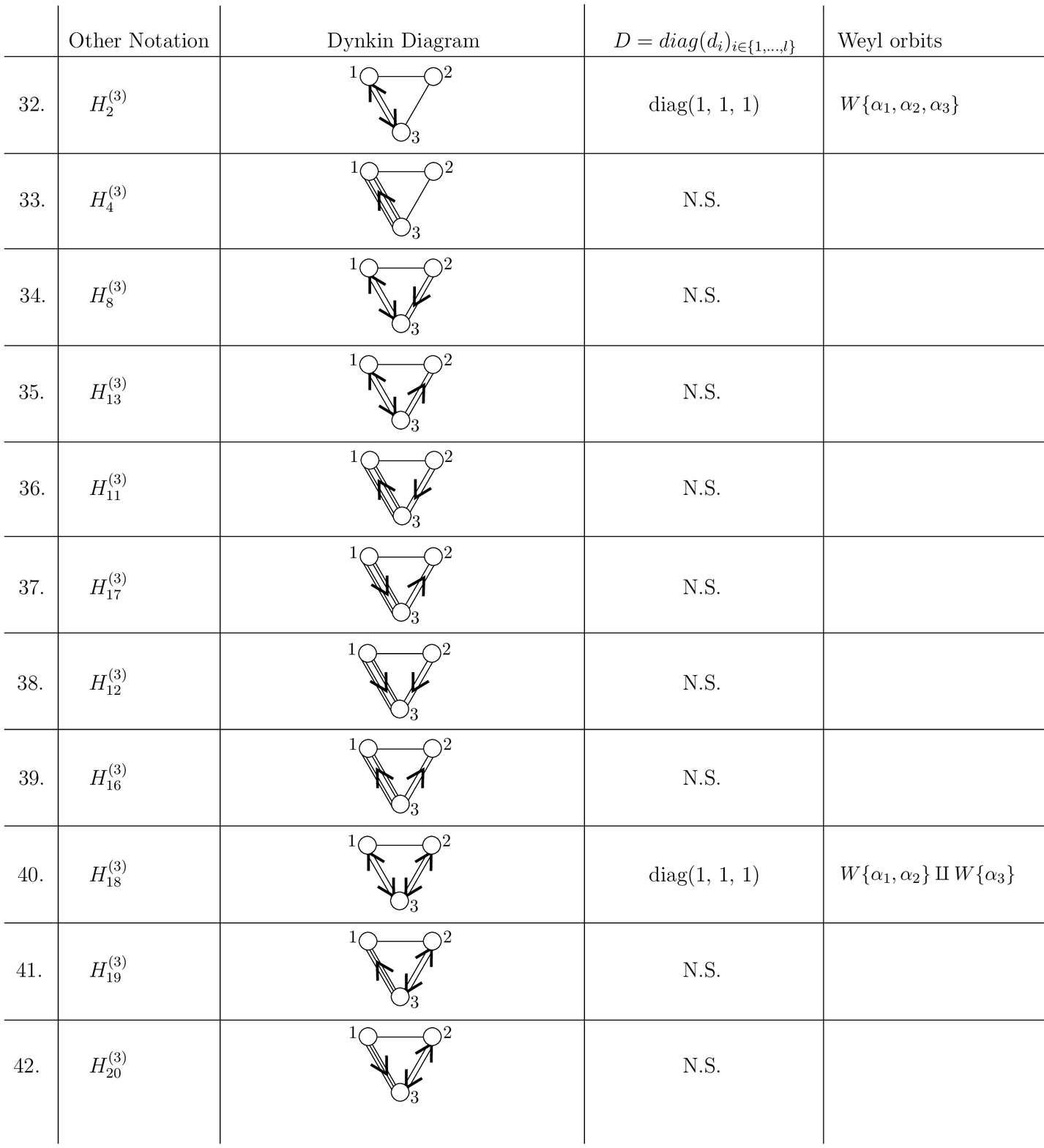}}
\end{table}

\newpage

\begin{table}[h!]
  \caption{Rank 3 non-compact diagrams (continued)}
  \centering
\resizebox{6.5in}{6.5 in }{\includegraphics*[viewport=35 35 720 720 ]{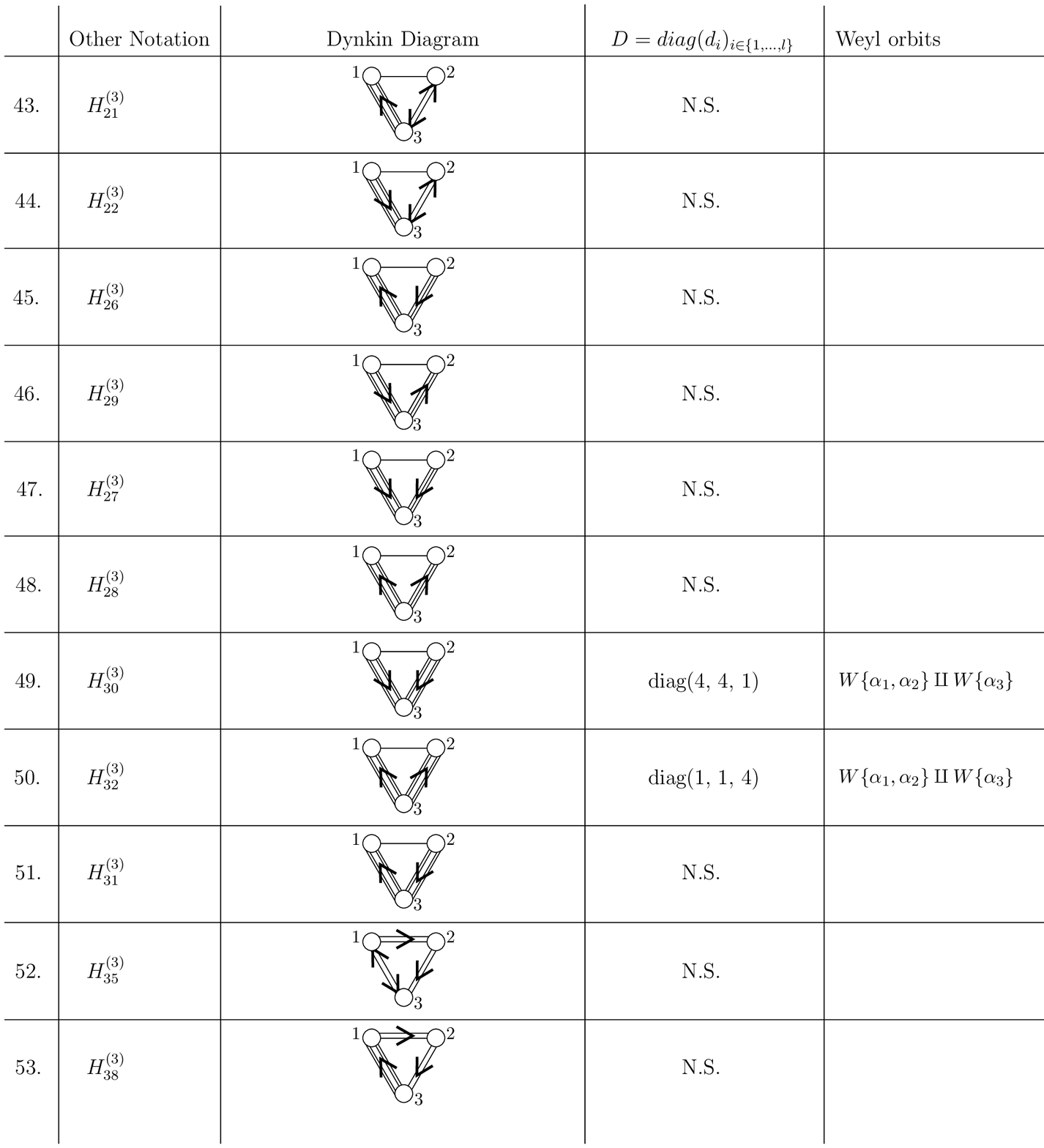}}
\end{table}

\newpage

\begin{table}[h!]
  \caption{Rank 3 non-compact diagrams (continued)}
  \centering
\resizebox{6.5 in}{6.5 in }{\includegraphics*[viewport=30 30 720 720 ]{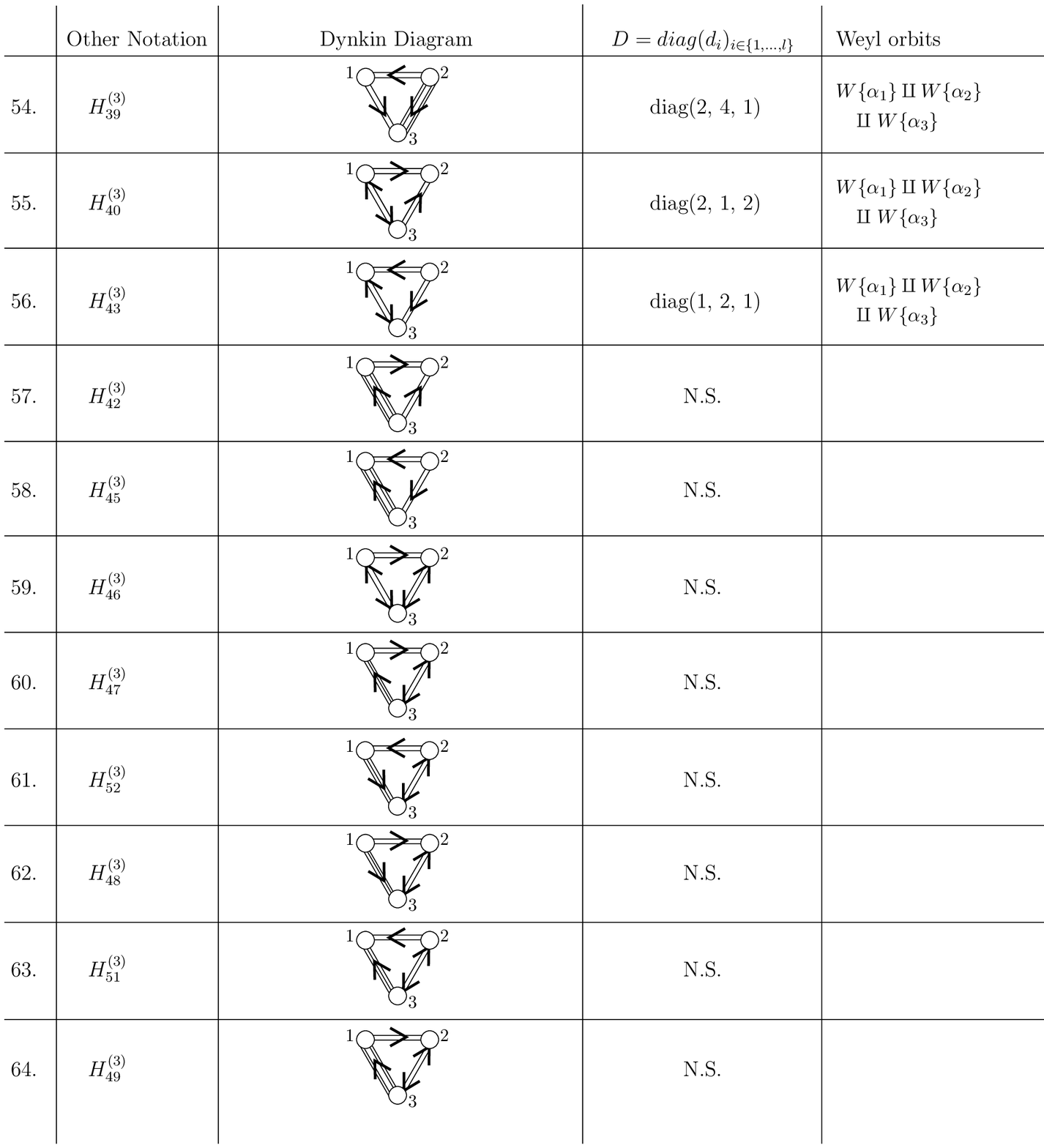}}
\end{table}

\newpage

\begin{table}[h!]
  \caption{Rank 3 non-compact diagrams (continued)}
  \centering
\resizebox{6.5in}{6.5 in }{\includegraphics*[viewport=35 35 720 720 ]{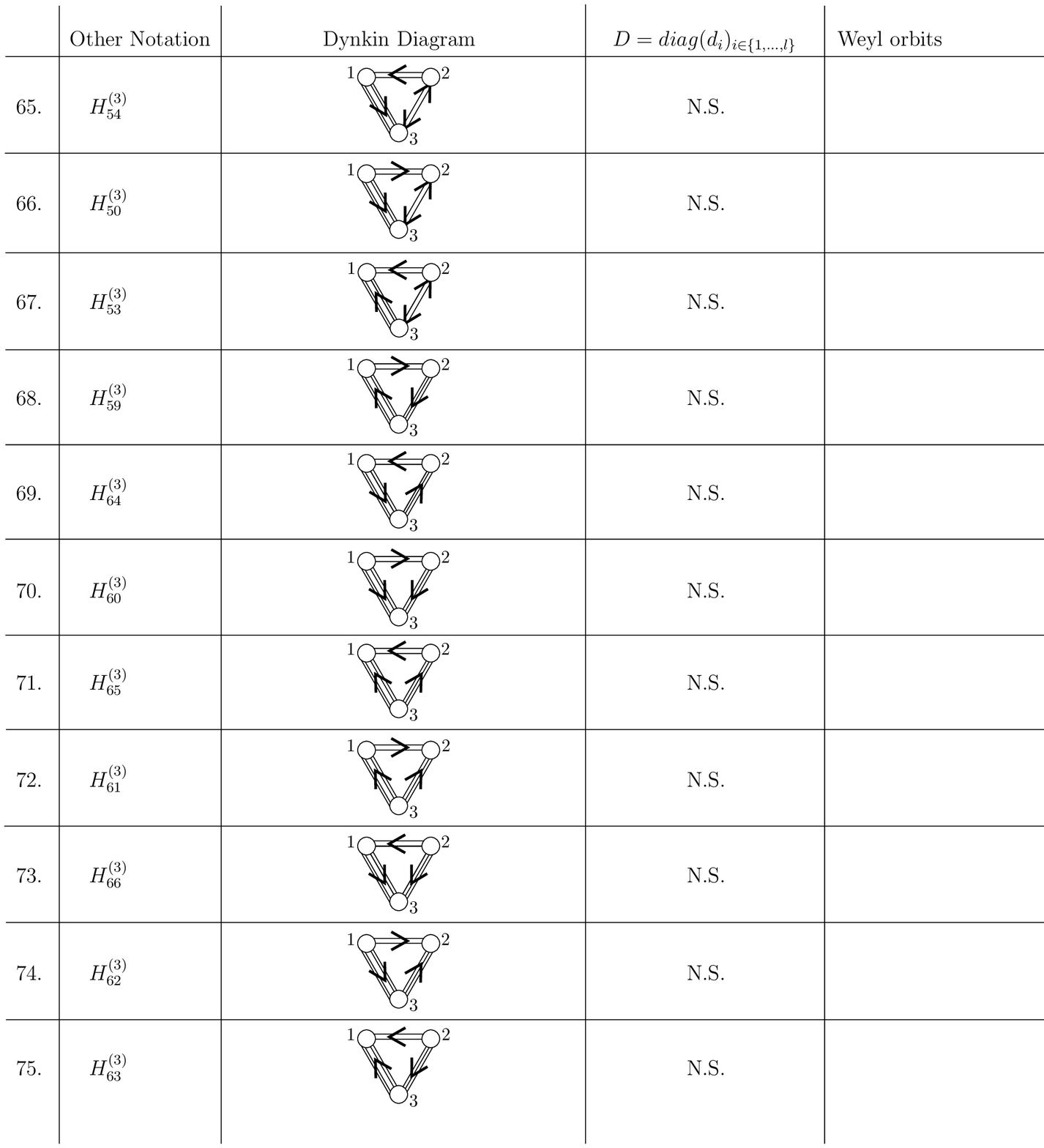}}
\end{table}

\newpage

\begin{table}[h!]
  \caption{Rank 3 non-compact diagrams (continued)}
  \centering
\resizebox{6.5in}{6.5 in }{\includegraphics*[viewport=35 35 720 720 ]{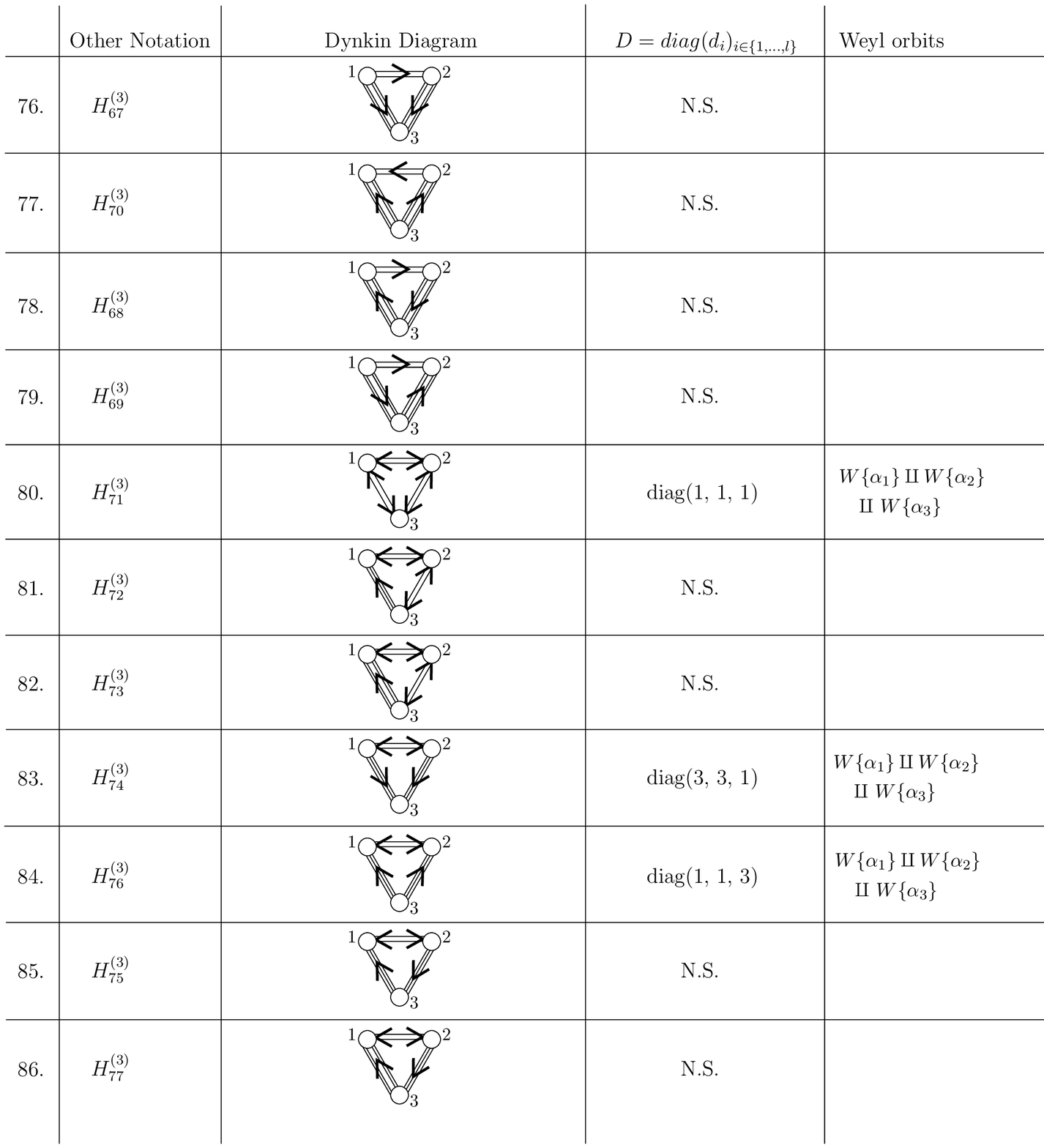}}
\end{table}

\newpage

\begin{table}[h!]
  \caption{Rank 3 non-compact diagrams (continued)}
  \centering
\resizebox{6.5in}{6.5 in }{\includegraphics*[viewport=35 35 720 720 ]{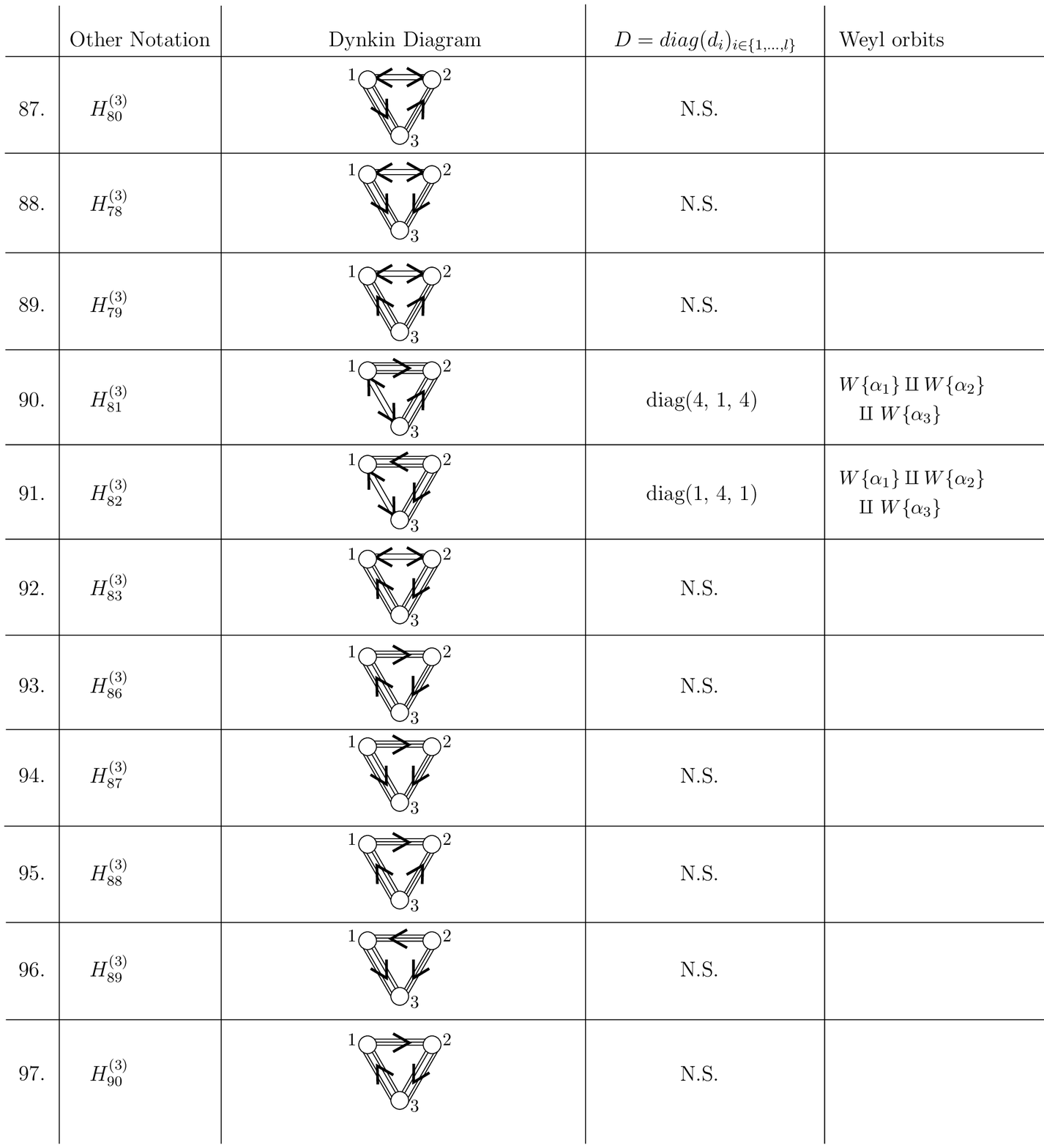}}
\end{table}

\newpage

\begin{table}[h!]
  \caption{Rank 3 non-compact diagrams (continued)}
  \centering
\resizebox{6.5in}{6.5 in }{\includegraphics*[viewport=30 30 720 720 ]{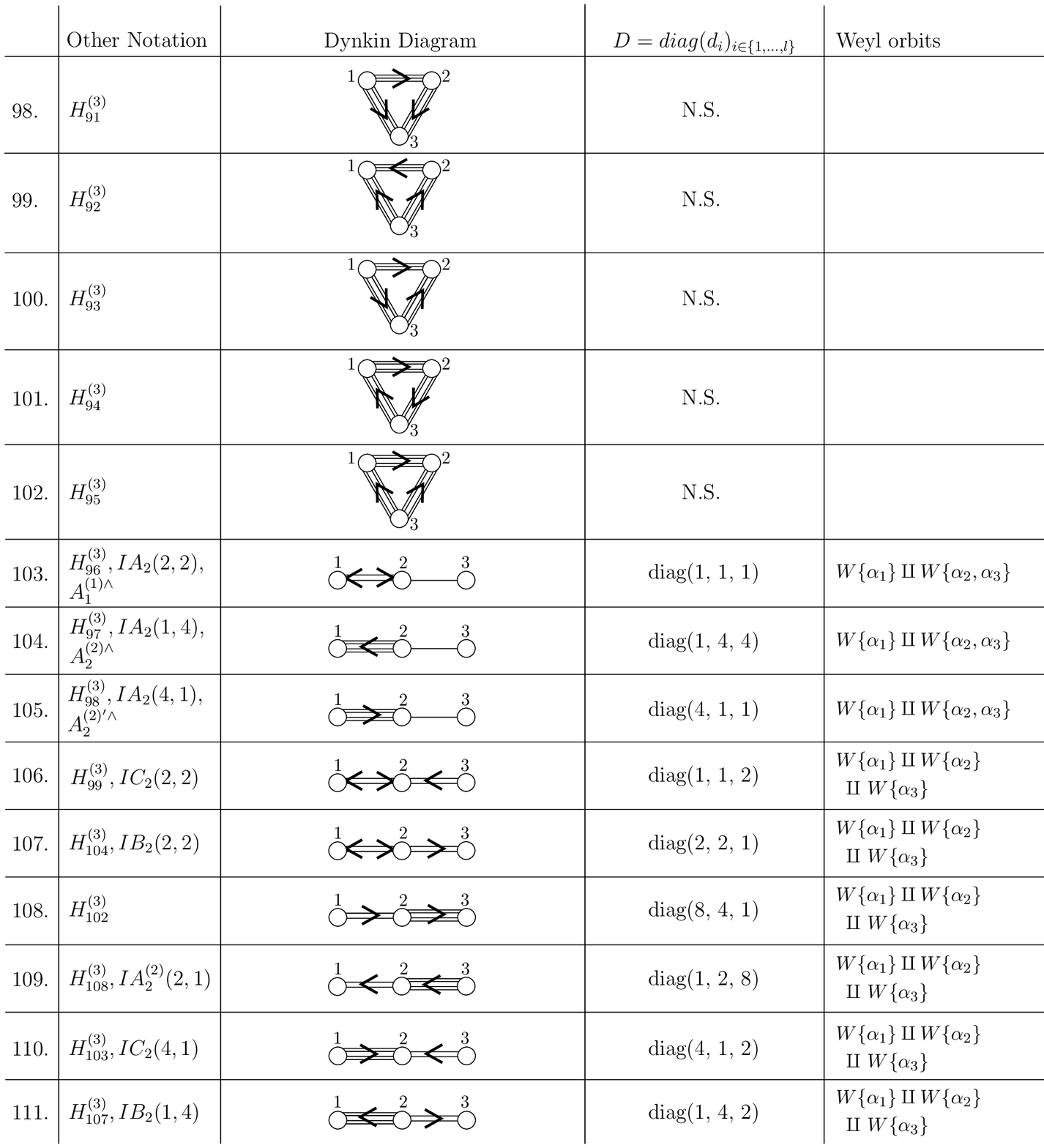}}
\end{table}

\newpage

\begin{table}[h!]
  \caption{Rank 3 non-compact diagrams (continued)}
  \centering
\resizebox{6.5in}{6.5 in }{\includegraphics*[viewport=35 35 720 720 ]{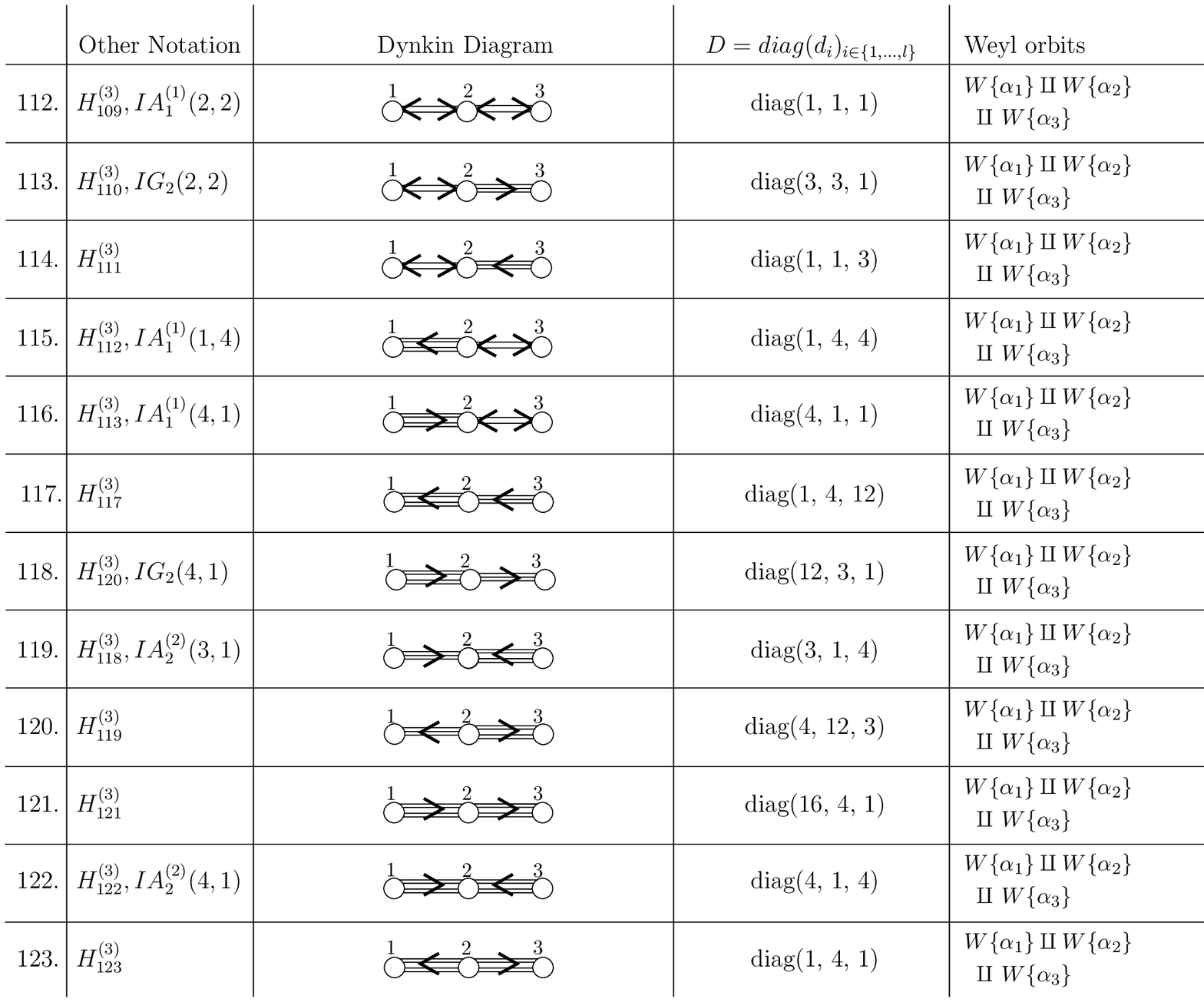}}
\end{table}

\newpage

\begin{table}[h!]
  \caption{Rank 4 diagrams}
  \vspace{-8.0 pt}
  \tiny{\it{(All diagrams are of non-compact type unless otherwise noted.)}}
  \centering
  
\vspace{12.0 pt}
\resizebox{6.5in}{6.5 in }{\includegraphics*[viewport=30 30 720 720 ]{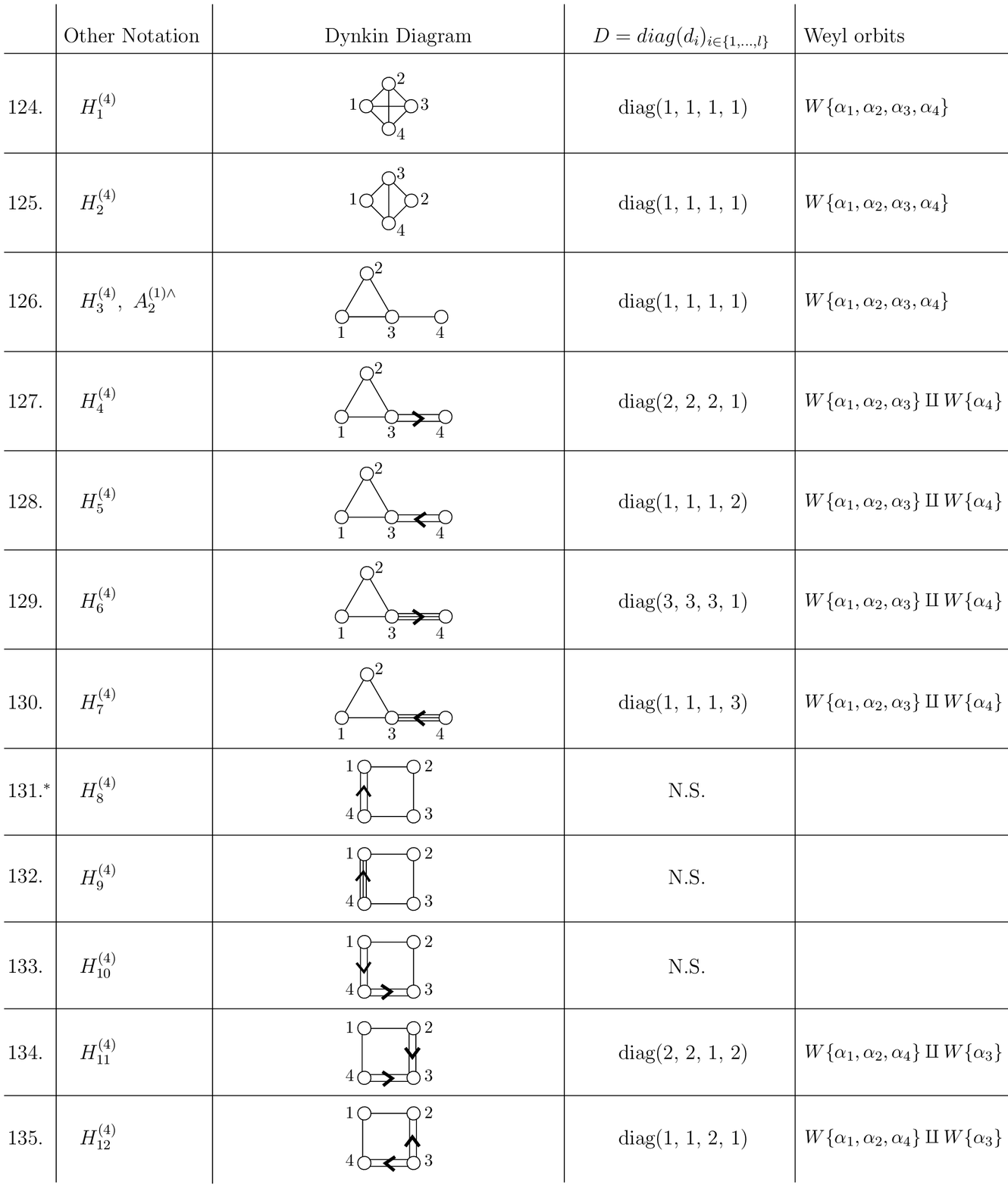}}
\end{table}
\textcolor{white}{\symbolfootnote[1]{ \ This diagram is compact.}} 
\newpage

\begin{table}[h!]
  \caption{Rank 4 diagrams (continued)}
  \vspace{-8.0 pt}
  \tiny{\it{(All diagrams are of non-compact type unless otherwise noted.)}}
  \centering
  
\vspace{12.0 pt}
\resizebox{6.5in}{6.5 in }{\includegraphics*[viewport=30 30 720 720 ]{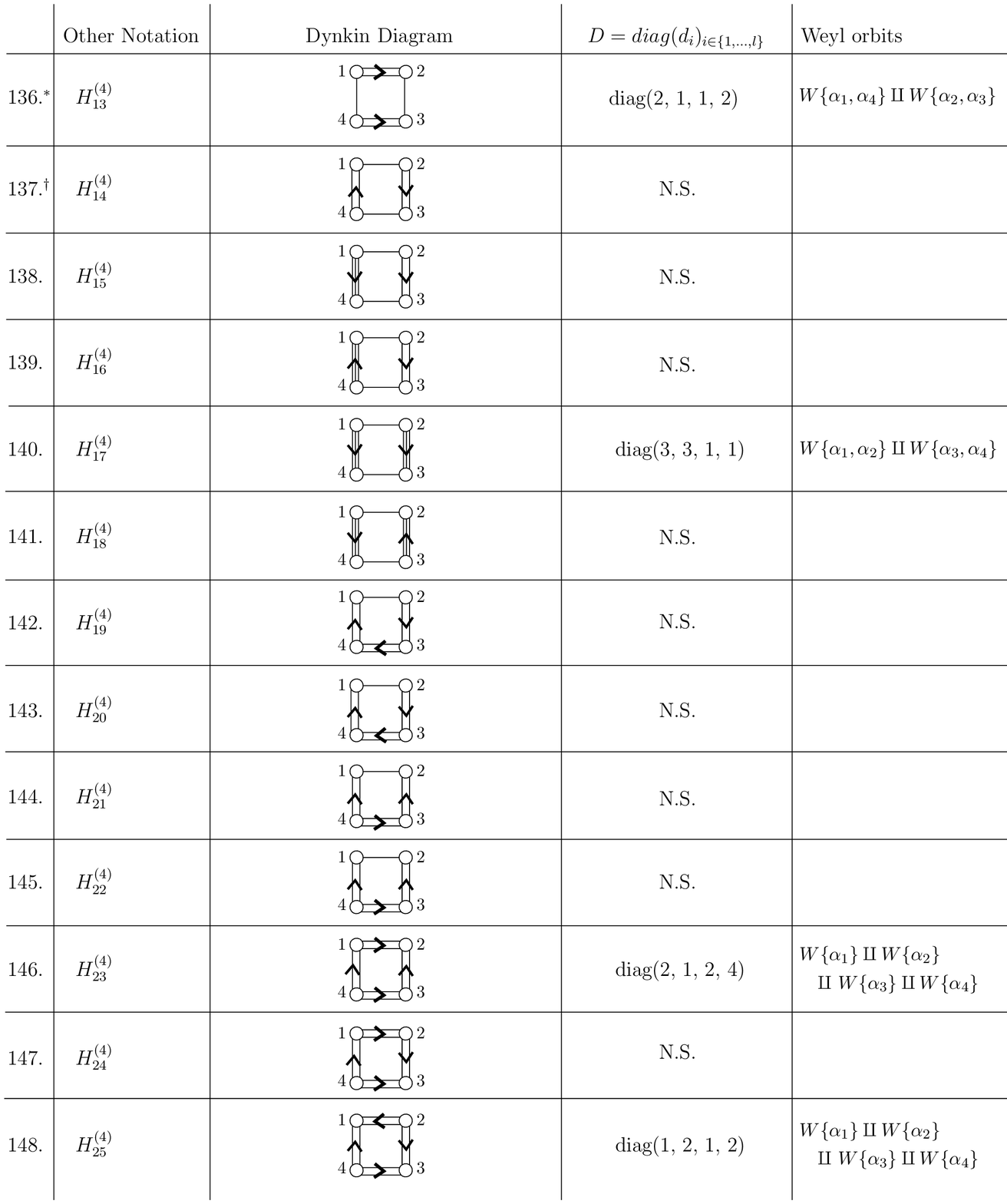}}
\end{table}
\textcolor{white}{\symbolfootnote[1]{$^{, \dagger}$ \ These diagrams are compact.}} 

\newpage

\begin{table}[h!]
  \caption{Rank 4 diagrams (continued)}
  \vspace{-8.0 pt}
  \tiny{\it{(All diagrams are of non-compact type unless otherwise noted.)}}
  \centering
  
\vspace{12.0 pt}
\resizebox{6.5in}{6.5 in }{\includegraphics*[viewport=35 35 720 720 ]{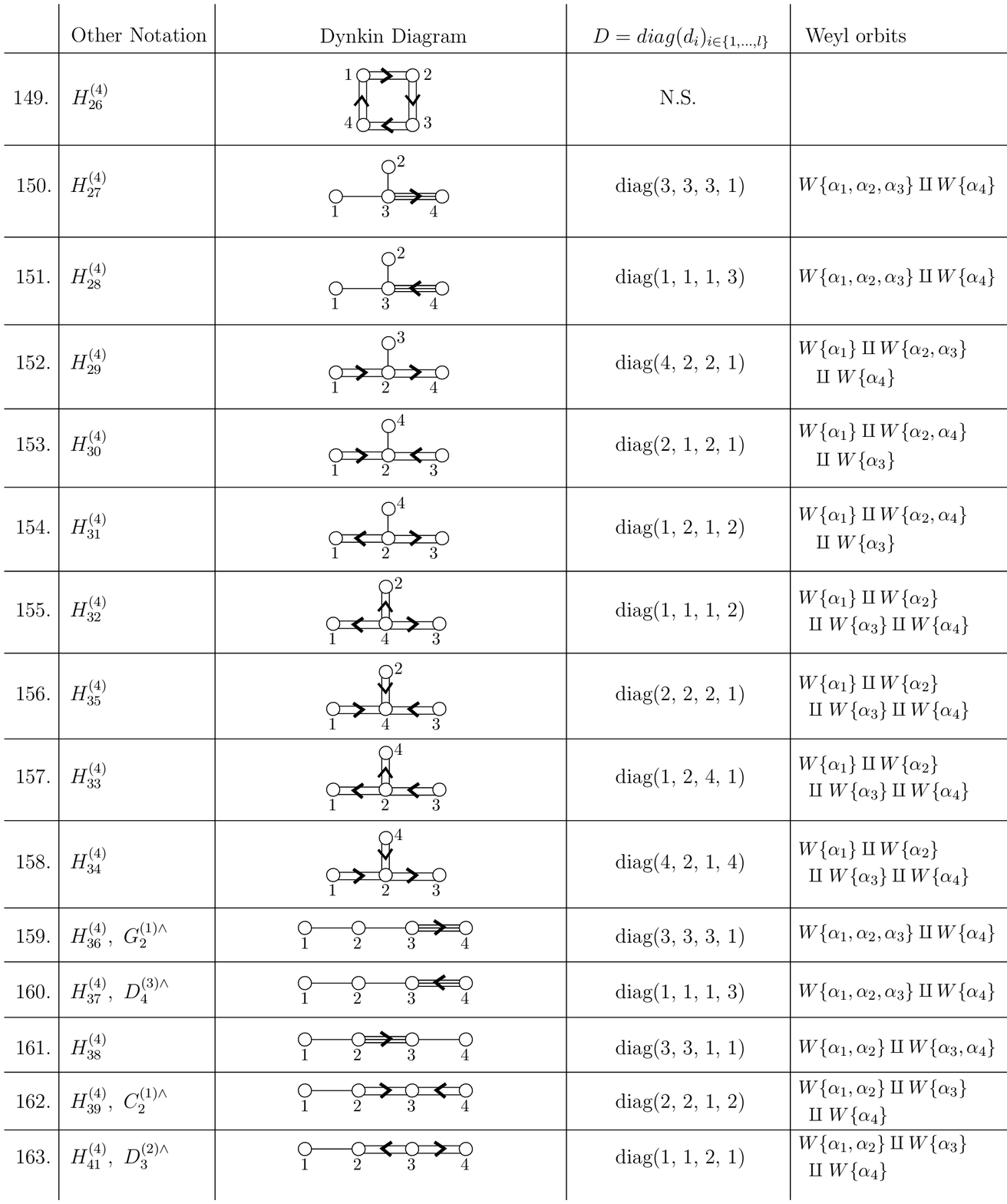}}
\end{table}

\newpage

\begin{table}[h!]
  \caption{Rank 4 diagrams (continued)}
  \vspace{-8.0 pt}
  \tiny{\it{(All diagrams are of non-compact type unless otherwise noted.)}}
  \centering
  
\vspace{12.0 pt}
\resizebox{6.5in}{6.5 in }{\includegraphics*[viewport=35 35 720 720 ]{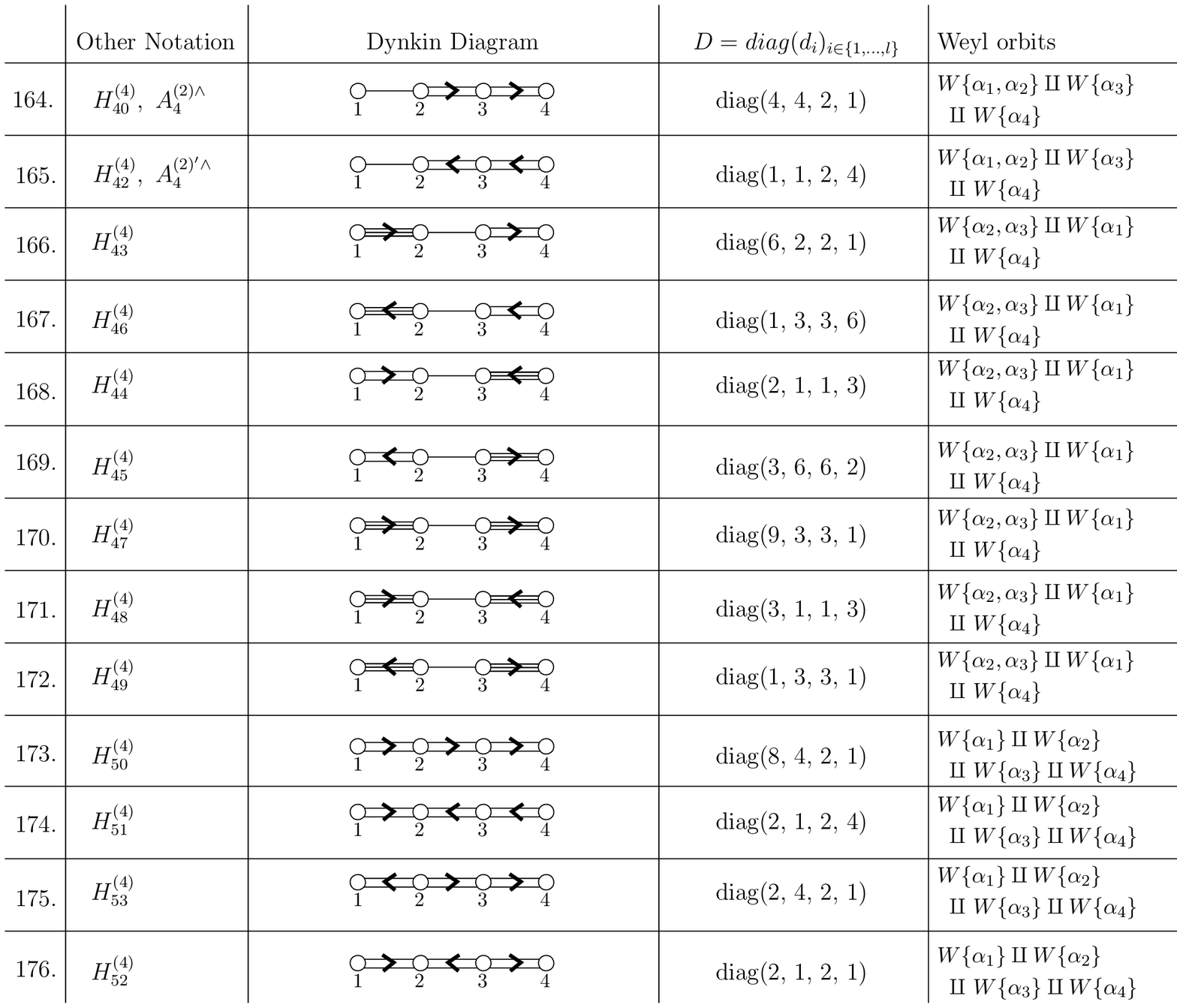}}
\end{table}

\newpage

\begin{table}[h!]
  \caption{Rank 5 diagrams}
  \vspace{-8.0 pt}
  \tiny{\it{(All diagrams are of non-compact type unless otherwise noted.)}}
  \centering
  
\vspace{12.0 pt}
\resizebox{6.5in}{6.5 in }{\includegraphics*[viewport=30 30 720 720 ]{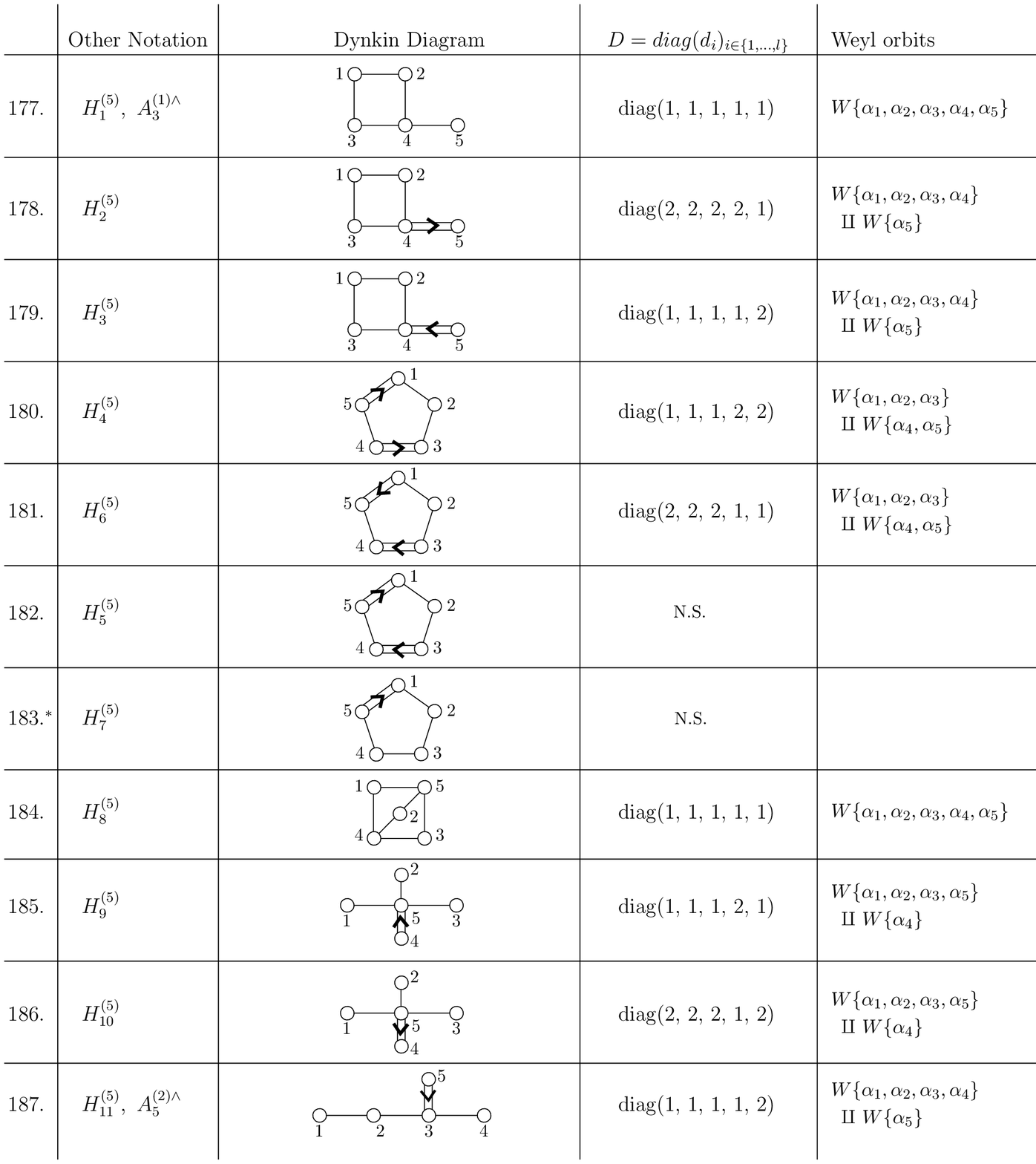}}
\end{table}
\textcolor{white}{\symbolfootnote[1]{ \ This diagram is compact.}}

\newpage

\begin{table}[h!]
  \caption{Rank 5 diagrams (continued)}
  \vspace{-8.0 pt}
  \tiny{\it{(All diagrams are of non-compact type unless otherwise noted.)}}
  \centering
  
\vspace{12.0 pt}
\resizebox{6.5in}{6.5 in }{\includegraphics*[viewport=35 35 720 720 ]{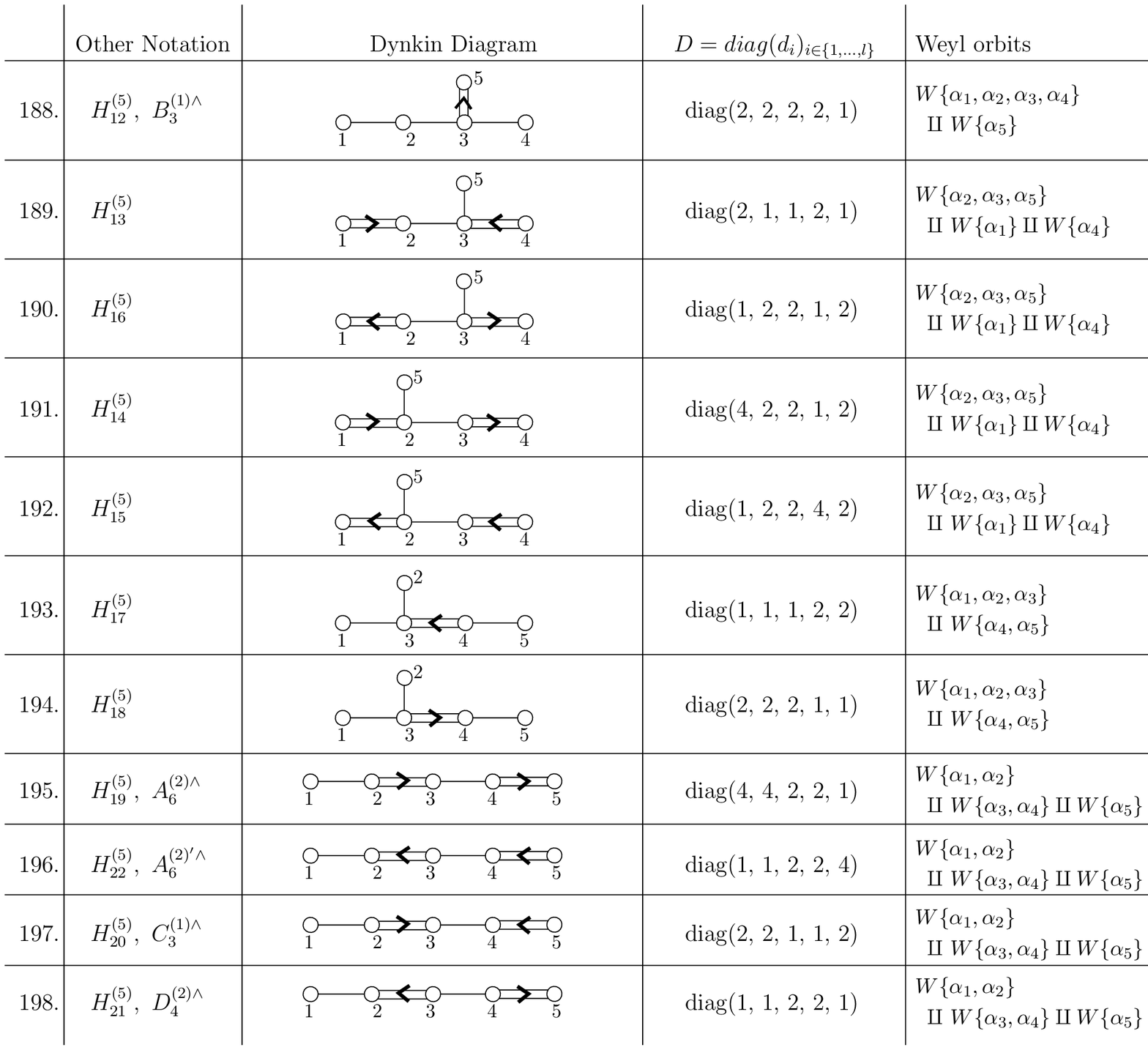}}
\end{table}

\newpage

\begin{table}[h!]
  \caption{Rank 6 diagrams}
  \centering
\resizebox{6.5in}{6.5 in }{\includegraphics*[viewport=35 35 720 720 ]{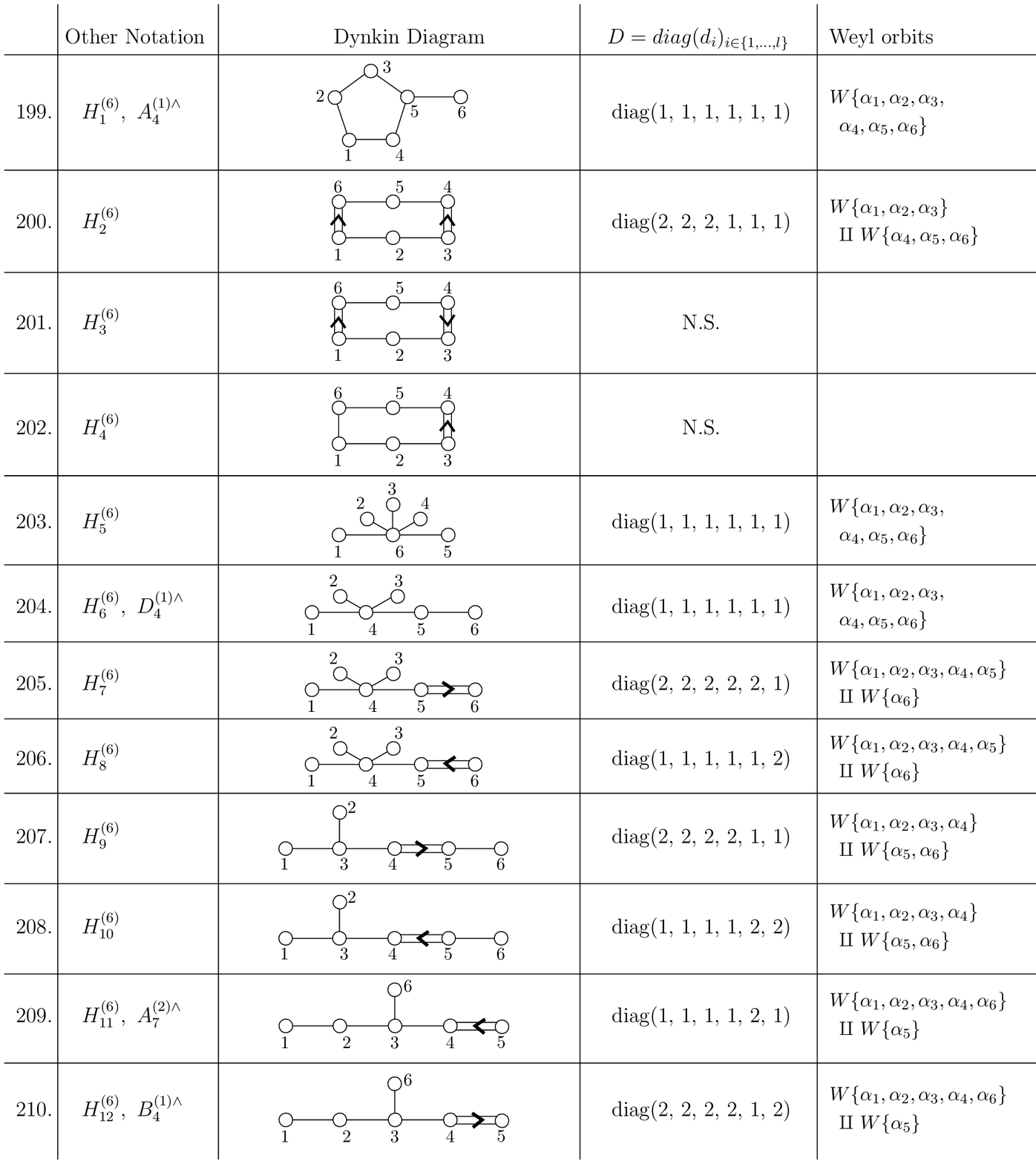}}
\end{table}

\newpage

\begin{table}[h!]
  \caption{Rank 6 diagrams (continued)}
  \centering
\resizebox{6.5 in}{4.11 in }{\includegraphics*[viewport=35 287 720 720 ]{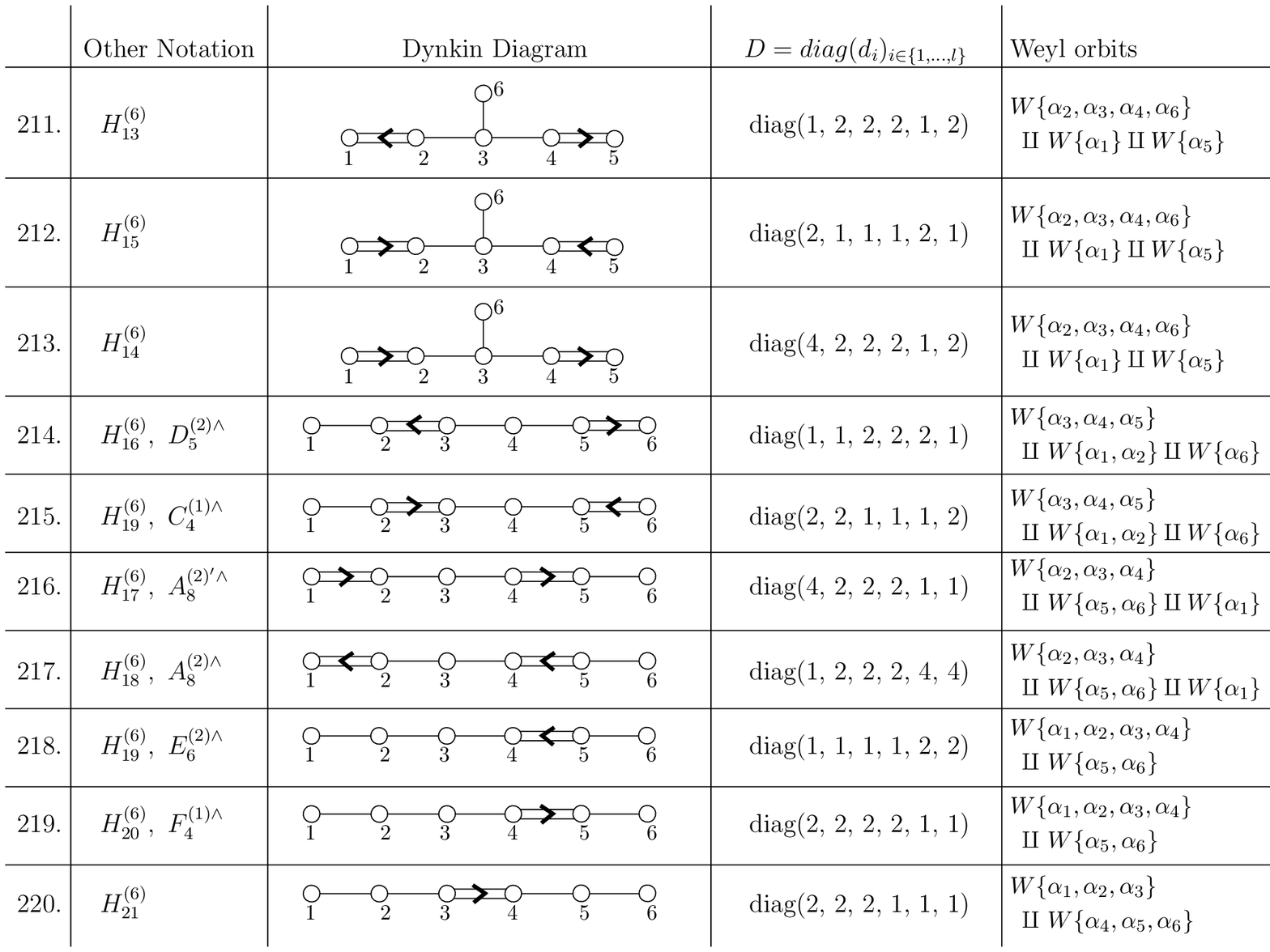}}
\end{table}

\newpage

\begin{table}[h!]
  \caption{Rank 7 diagrams}
  \centering
\resizebox{6.5 in}{2.5 in }{\includegraphics*[viewport=35 460 720 720 ]{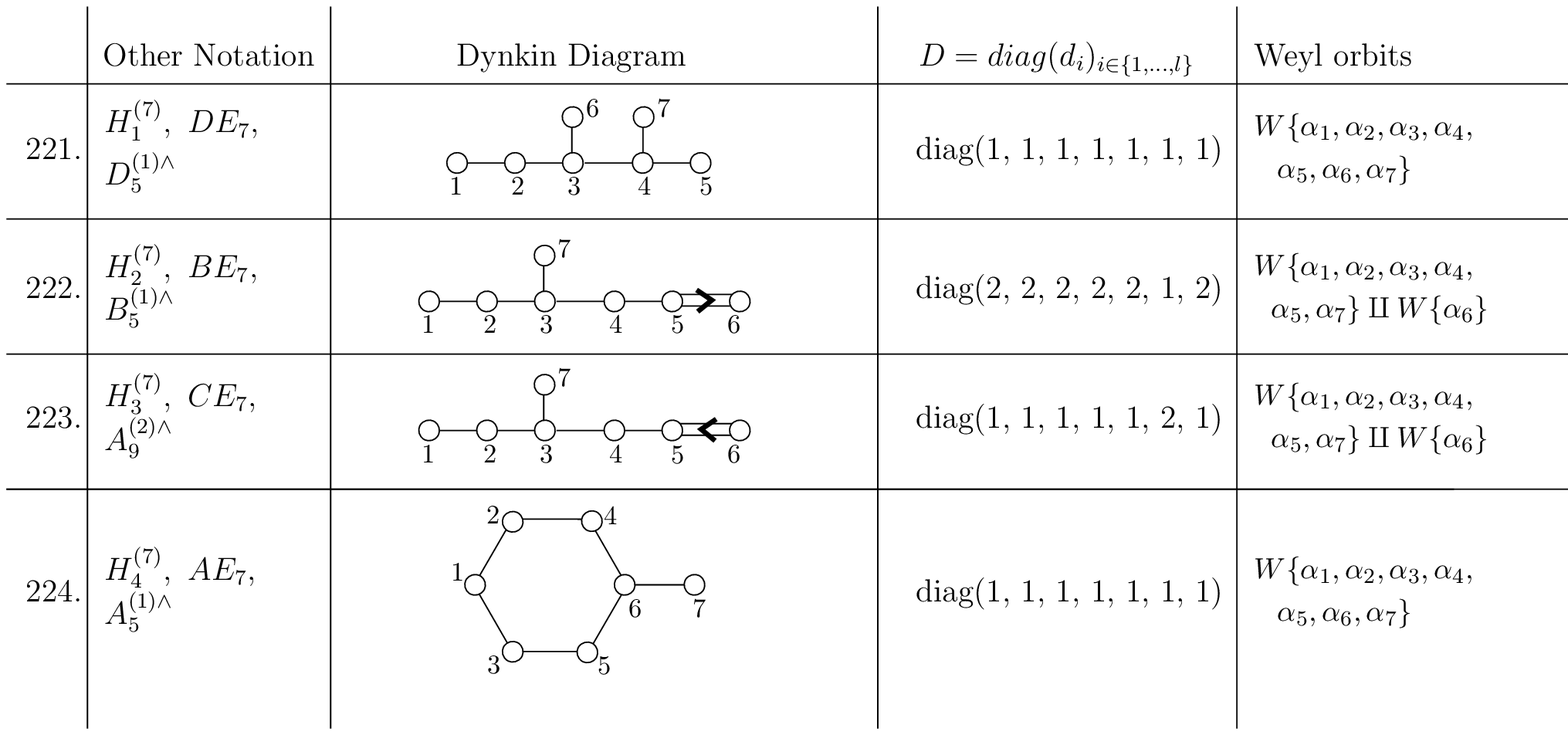}}
\end{table}

\begin{table}[h!]
  \caption{Rank 8 diagrams}
  \centering
\resizebox{6.5in}{3.5 in }{\includegraphics*[viewport=35 360 720 720 ]{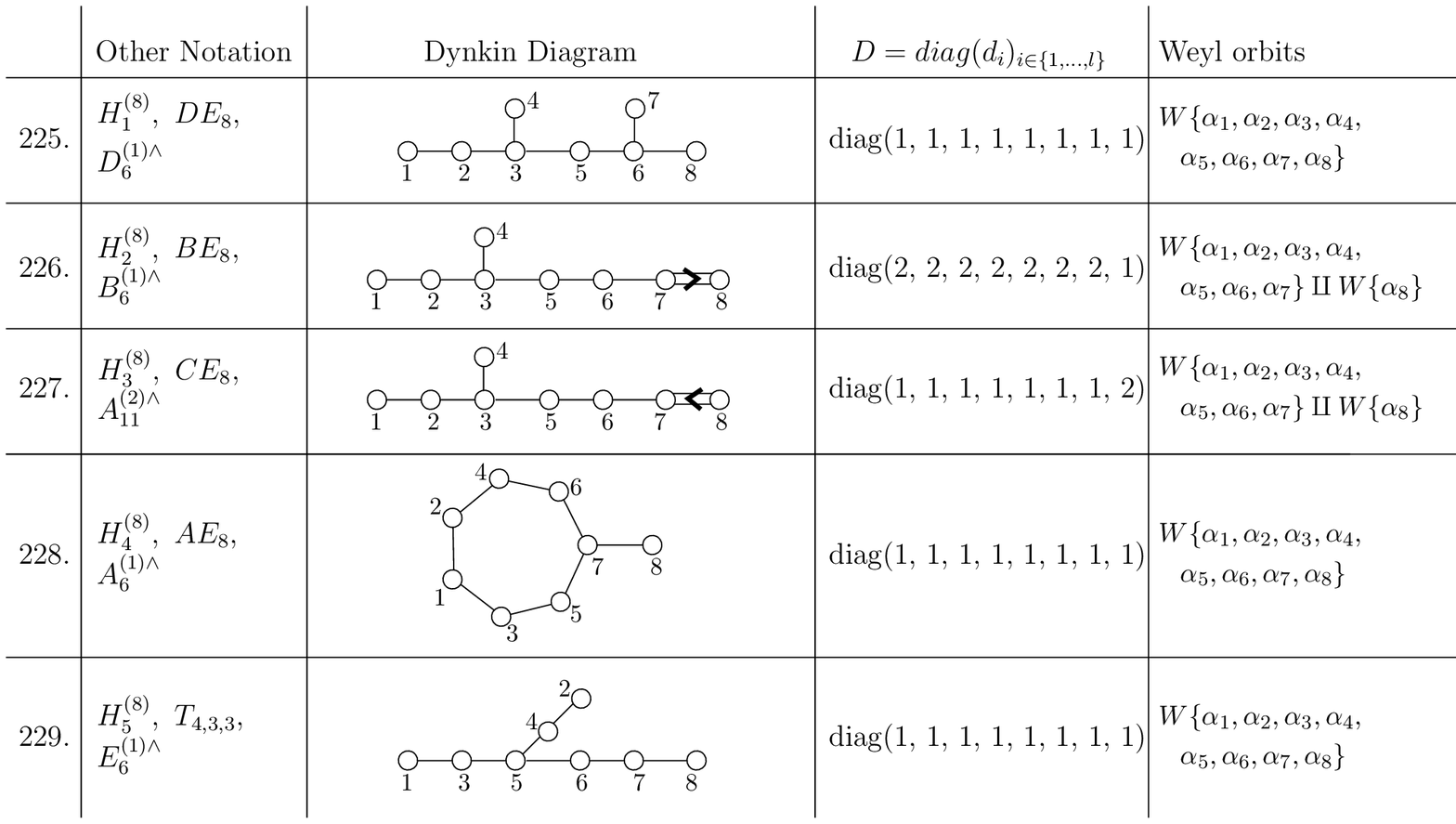}}
\end{table}

\newpage

\begin{table}[h!]
  \caption{Rank 9 diagrams}
  \centering
\resizebox{6.5in}{3.04 in }{\includegraphics*[viewport=35 400 720 720 ]{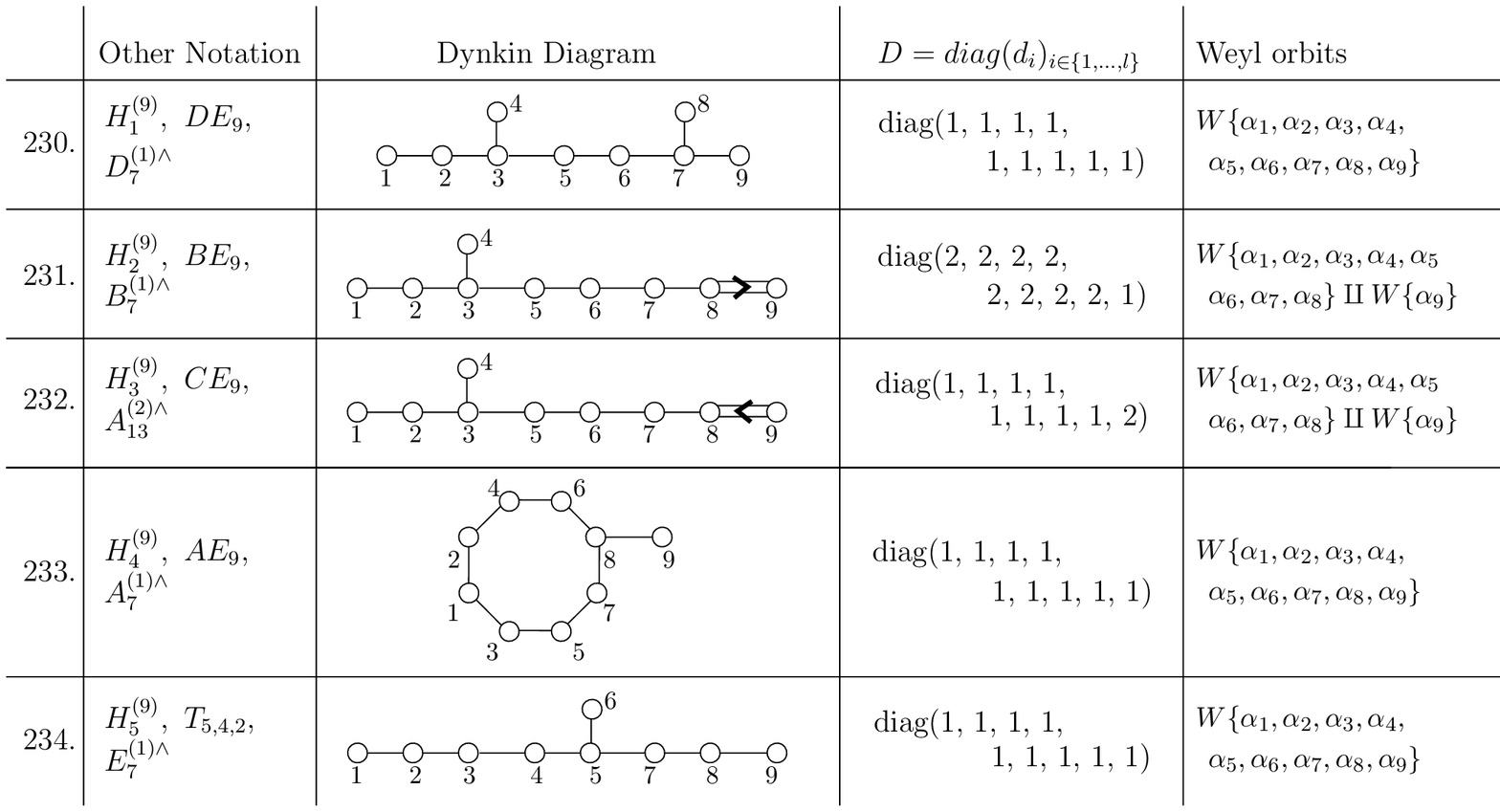}}
\end{table}

\begin{table}[h!]
  \caption{Rank 10 diagrams}
  \centering
\resizebox{6.5in}{2.5 in }{\includegraphics*[viewport=35 460 720 720 ]{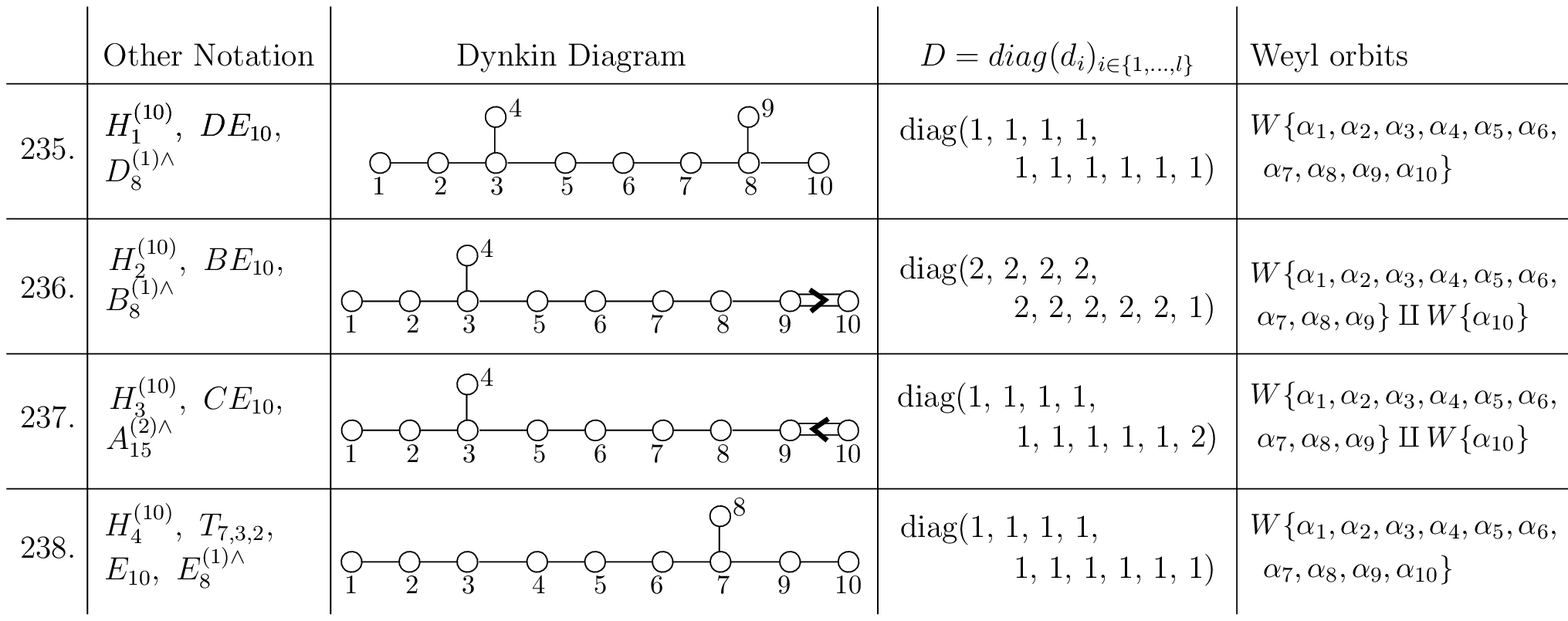}}
\end{table}

\newpage

\medskip\noindent
\section{Errata in the existing literature}
In this section we summarize the errata in the existing literature. The tables of [KM] appear not to have any errata.

\medskip\noindent {\bf Errata in the paper of Li ([Li]) }

\medskip\noindent (1)  $H_{21}^{(3)}$ should have a double arrow on the right.  Without this double arrow, the diagram is a copy of   $H_{11}^{(3)}$. In fact  $H_{21}^{(3)}$ should be the dual of $H_{22}^{(3)}$ which is missing.

\medskip\noindent (2) $H_{73}^{(3)}$ should have double arrow on right.  Without this double arrow, the diagram is a copy of $H_{54}^{(3)}$. 
	
\medskip\noindent {\bf Errata in the paper of Sa\c{c}lio\~{g}lu ([Sa])}

\medskip\noindent (i)  Table 2, Rank 3, No. 18, dual diagram should have a double arrow between vertices 2 and 3. This corresponds to diagram no. 116 of Section 7.

\medskip\noindent (ii)  Symmetric diagram no. 29 of  Section 7 is missing (noted in [dBS]).
  
\medskip\noindent (iii)   Symmetric diagram no. 136 of Section 7 is missing (noted in [dBS]).
 
\medskip\noindent (iv)   Symmetric diagram no. 146 of Section 7 is missing (noted in [dBS]).

\medskip\noindent (v)  Symmetric diagram no. 170 of Section 7 is missing (noted in [dBS]).
  
\medskip\noindent (vi)  Symmetric diagram no. 197 of Section 7 is missing (noted in [dBS]).
 
\medskip\noindent (vii)  Symmetric diagram no. 198 of Section 7 is missing (noted in [dBS]).

\medskip\noindent {\bf Errata in the book of Wan ([W] )}

\medskip
There are numerous misprints in Section $2.6$ of this book that presents an account of Li's classification. For example, there are many missing edges and edge orientations. We list a few of these below.

\medskip\noindent (i)  $H_{21}^{(3)}$ should have 4 edges on left to match [Li], and should have double arrow on right as per erratum (1) of [Li] above.

\medskip\noindent (ii)  	$H_{73}^{(3)}$ is incorrect as per erratum (2) of [Li]  above.  (This is not a misprint, but a copy of Li's error.) 

\medskip\noindent (iii)  	$H_1^{(3)}$ is missing an arrow on left.

\medskip\noindent (iv)  $H_{90}^{(3)}$ should have 4 edges on right.

\medskip\noindent (v)  $H_{27}^{(4)}$- $H_{31}^{(4)}$ are all missing an edge.

\medskip\noindent (vi)  $H_{32}^{(4)}$- $H_{34}^{(4)}$ are all missing an upward 2-arrow.

\medskip\noindent (vii)  $H_{35}^{(4)}$ is missing a downward 2-arrow.

\end{document}